\documentclass[12pt,a4paper,oneside]{article}
\usepackage[top=3cm, bottom=3cm, left=2cm, right=2cm]{geometry}
\usepackage[english]{babel}
\usepackage[utf8]{inputenc}
\usepackage[T1]{fontenc}

\usepackage{amsthm,amsmath,amssymb,mathrsfs,dsfont,mathtools}
\usepackage{bbm}
\usepackage{bm}
\usepackage{graphicx}
\usepackage[colorlinks=true, allcolors=blue]{hyperref}
\usepackage{comment}
\usepackage[all]{nowidow}
\usepackage{booktabs}
\usepackage{cleveref}
\usepackage{float}
\usepackage{multirow}
\usepackage{algorithm}
\usepackage{algorithmic}
\usepackage{tikz}
\usepackage{pgfplots}
\usepackage{makecell}
\usepgfplotslibrary{groupplots}
\usepackage{subcaption}
\allowdisplaybreaks

\usepackage[authoryear, round]{natbib}

\newcommand{\iid}{\stackrel{\mbox{\scriptsize iid}}{\sim}}

\newcommand{\indicator}{\ensuremath{\mathbbm{1}}}

\renewcommand{\mid}{\ensuremath{\,|\,}}

\usepackage{authblk}
\date{}
\title{Bayesian Nonparametric Privacy-Preserving Synthetic Data Generation: I. Discrete Data}
\author[1]{Maria Chiara Menicucci}
\author[2]{Mario Beraha}
\author[1]{Stefano Favaro}
\author[3]{Riccardo Lazzarini}

\affil[1]{Department of Economics and Statistics, University of Torino and Collegio Carlo Alberto, 10134 Torino, Italy}
\affil[2]{Department of Economics, Management, and Statistics, University of Milano--Bicocca, 20126 Milano, Italy}
\affil[3]{Department of Decision Sciences, Bocconi University, 20136 Milano, Italy}

\providecommand{\keywords}[1]{
  \small 
  \textbf{\textit{Keywords:}} #1
  \normalsize
}

\usepackage[sf]{titlesec}
\usepackage{sectsty}
\allsectionsfont{\centering \normalfont\scshape} % Make all sections centered, the default font and small caps

\newtheorem{theorem}{Theorem}
\newtheorem{corollary}{Corollary}
\newtheorem{lemma}{Lemma}
\newtheorem{proposition}{Proposition}
\theoremstyle{definition}

\newtheorem{definition}{Definition}

\begin{document}

\maketitle

% \tableofcontents

\begin{abstract}
Synthetic data generation is a powerful approach to privacy-preserving statistical analysis, where data-release mechanisms are governed by a privacy-utility tradeoff: they should provide privacy guarantees while preserving the statistical utility of confidential data. We develop a Bayesian nonparametric framework for private synthetic data generation tailored to discrete data. Specifically, the confidential data are modeled as a random sample from an unknown discrete distribution endowed with a Pitman-Yor process prior, and synthetic data are generated from the corresponding posterior-predictive distribution. Since the Pitman-Yor process defines an almost surely discrete random probability measure, the resulting mechanism is naturally suited to data with ties and settings involving a potentially large, unknown, or growing number of categories. We study differential privacy guarantees of the Pitman-Yor posterior-predictive mechanism across the three regimes of the discount parameter $\sigma\in(-\infty,1)$. For $\sigma\in(0,1)$, we establish an instance-level $(\varepsilon,\delta)$-differential privacy guarantee. For $\sigma=0$ and $\sigma<0$, corresponding respectively to the Dirichlet process prior and to a parametric Dirichlet-Multinomial model, stronger guarantees are obtained, under suitable conditions on the released sample size. We also investigate statistical utility, or informativity, of the released data via the expected $1$-Wasserstein distance between the empirical distribution of the synthetic data and the "true" data-generating distribution. For $\sigma<0$ and $\sigma=0$, we prove consistency of the empirical distribution in this metric and derive explicit convergence rates, making precise the privacy-utility tradeoff: stronger privacy guarantees impose more restrictive choices of the released sample size, slowing down convergence to the "true" data-generating distribution.
\end{abstract}

\keywords{Bayesian nonparametrics; differential privacy; Dirichlet process prior; mechanism informativity, Pitman--Yor process prior; synthetic data generation; $1$-Wasserstein posterior contraction rates
}

%%%%%%%%%%%%%%%%%%%%%%%%%%%%%%%%
%%%%%%%%%%%%%%%%%%%%%%%%%%%%%%%%
%%%%%%%%%%%%%%%%%%%%%%%%%%%%%%%%
%%%%%%%%%%%%%%%%%%%%%%%%%%%%%%%%

\section{Introduction}

\subsection{Background and motivation}

Over the past three decades, synthetic data generation has become a powerful approach to privacy-preserving statistical analysis, aiming to reconcile data accessibility with the protection of sensitive information \citep{rubin1993statistical,DH24}. In this framework, a statistical model is fitted to the original confidential data and then used to draw new observations. These simulated or synthetic observations are released in place of the original data, which remain confidential. The effectiveness of such a synthetic data-generating mechanism is governed by a privacy--utility tradeoff: synthetic data should protect sensitive information from disclosure, while preserving enough of the statistical information in the original data to support downstream analyses. 

On the privacy side, differential privacy (DP) has become the standard framework for providing formal guarantees of data-release mechanisms. Introduced by \cite{differentialprivacy}, DP is now widely adopted in both academia and industry \citep{erlingsson,apple2017learning,ding2017collecting,abowd2022census}. At a high level, a mechanism is differentially private if changing a single database record has only a limited effect on its output. On the utility side, once a data-release mechanism has been shown to satisfy DP, it remains essential to assess whether the privatization process preserves enough statistical information from the original data. In this sense, utility requires that the released data remain informative for inference on the confidential data.  A general notion of statistical utility for data-release mechanism was proposed by \cite{wasserman}, under the name of informativity or consistency. This notion adopts the frequentist perspective that the confidential data are generated from an underlying ``true'' distribution, and measures utility through the rate at which the expected distance between the empirical measure of the released data and the true data-generating distribution converges to zero. Such a convergence rate provides a natural measure of the quality of the mechanism, quantifying how quickly the released data recover the underlying distribution.

Many private synthetic data-generating mechanisms have been proposed in the literature, ranging from histogram-based methods and exponential-mechanism to more recent approaches based on private-measure release and optimal transport \citep{machana,wasserman,mcsherry2007mechanism,boedihardjo2024privatemeasuresrandomwalks,He_2024,he2023algorithmicallyeffectivedifferentiallyprivate}. Within the Bayesian framework, a natural strategy is to release synthetic observations drawn from a posterior or posterior-predictive distribution \citep{dimitrakakis,Hu2022MechanismsFG,bernstein2019differentially,
NEURIPS2023_f3024ea8,huetal,machana}. This is attractive because the data-release mechanism is directly linked to the statistical model used for inference. Specifically, suppose that $n\geq1$ confidential observations are modeled through a Bayesian model with sampling distribution $p(\cdot\mid\theta)$ and prior distribution $\pi$:
\begin{equation}\label{eq:bayesmod}
X_1,\ldots,X_n \mid \theta \overset{\mathrm{iid}}{\sim} p(\cdot \mid \theta),\qquad \theta \sim \pi,
\end{equation}
and that $m\geq1$ observations $Z_1,\ldots,Z_m$ are drawn from the induced posterior-predictive distribution
\begin{displaymath}
Q_n(\mathrm dz_1,\dots,dz_m\mid X_1,\dots,X_n) = \int p^{\otimes m}(\mathrm dz_1,\dots,dz_m \mid \theta)\,\pi(\mathrm d\theta\mid X_1,\dots,X_n).
\end{displaymath}
Under the model \eqref{eq:bayesmod}, the random variables \((X_1,\ldots,X_n,Z_1,\ldots,Z_m)\) are jointly exchangeable, which means that the analyst who receive the synthetic data \((Z_1,\dots,Z_m)\) can analyze it assuming 
\begin{displaymath}
Z_1,\ldots,Z_m \mid \theta \overset{\mathrm{iid}}{\sim} p(\cdot \mid \theta),
\qquad \theta \sim \pi,
\end{displaymath}
without worrying about the privatization step. This contrasts with generic, externally specified DP data-release mechanisms, typically involving a perturbation of the confidential data. As emphasized by \citet{beraha2025mcmc}, valid likelihood-based or Bayesian inference from such released data typically requires integrating over the unobserved confidential sample, leading to intractable likelihoods.

Most existing Bayesian approaches to private synthetic data generation have been developed within parametric models. Such models may be appropriate when the released dataset is designed for a specific inferential task, but they are less satisfactory for general downstream use and, in particular, when the analyses to be performed are not known in advance. Indeed model misspecification may translate directly into a loss of statistical utility: features of the confidential data that are not represented by the chosen parametric family cannot be recovered from the released data. Bayesian nonparametric methods provide a natural alternative. By placing a prior directly on an infinite-dimensional space of probability distributions, they allow the effective complexity of the model to adapt to the data and grow with the sample size, while preserving the posterior and posterior-predictive structure that makes model-aligned release attractive.

\subsection{Preview of our contributions}

We develop a Bayesian nonparametric framework for private synthetic data generation tailored to discrete data. By discrete data, we mean observations that may exhibit ties and whose support may contain a large, unknown, or growing number of categories. We model $n\geq1$ confidential observations as a random sample from an unknown discrete distribution, which is endowed with a Pitman--Yor (PY) process prior \citep{perman1992sizebiased,pitmanyor} with discount parameter $\sigma\in(-\infty,1)$ and strength parameter $\theta>-\sigma$. The PY process is particularly natural for discrete data, since it defines a random probability measure with almost surely discrete realizations; consequently, samples from the PY process contain repeated observations with positive probability, while the posterior-predictive distribution allows for the emergence of new categories. Given the $n\geq 1$ confidential observations modeled as described above, we generate $m\geq1$ synthetic observations by sampling from the $m$-step posterior-predictive distribution induced by the PY process prior; we refer to this synthetic data-generating mechanism as the PY posterior-predictive mechanism. The PY process prior includes two important subcases or regimes: for $\sigma=0$ it reduces to the Dirichlet process prior \citep{dirichlet}, while for $\sigma<0$ it corresponds, at the sampling level, to a parametric Dirichlet--Multinomial model. The PY posterior-predictive mechanism is straightforward to implement, owing to the tractability of the PY posterior-predictive distribution \citep{Pitman1995ExchangeableAP,pitmanyor}. Moreover, the joint exchangeability of confidential data and synthetic data generated by the mechanism allows standard Markov chain Monte Carlo algorithms to be used for Bayesian inference.

We study DP guarantees of the PY posterior-predictive mechanism. The differentially private behaviour of the mechanism is governed by the discount parameter $\sigma$, which controls the prevalence of low-frequency categories in the confidential data. Since rare categories are more easily attributable to individual records, low-frequency categories, and singletons in particular, are the main source of privacy risk: if a value appears only once, then releasing that value in the synthetic sample may reveal the contribution of a single individual. In the regime $\sigma\in[0,1)$, we prove that the PY posterior-predictive mechanism satisfies an instance-level \((\varepsilon,\delta)\)-DP guarantee in the sense of \citet{Soria_Comas_2017}; this privacy guarantee should be viewed as a data-dependent relaxation of the \((\varepsilon,\delta)\)-DP of \citet{differentialprivacy}. In the regime $\sigma=0$, we establish a \((\varepsilon,\delta)\)-DP bound for the PY posterior-predictive mechanism, with \(\delta\) controlled explicitly by the ratio \(m/n\) between the synthetic and confidential sample sizes. Finally, in the regime $\sigma<0$, we prove that the PY posterior-predictive mechanism satisfies instance-level $\varepsilon$-DP and, under suitable assumptions on the confidential sample size $m$, it satisfies $\varepsilon$-DP \citep{machana}.

We also investigate the informativity of the PY posterior-predictive mechanism. We quantify informativity of the mechanism through the $1$-Wasserstein distance, in line with recent works on differentially private synthetic data-generating mechanisms \citep{boedihardjo2024privatemeasuresrandomwalks,donhauser2023certifiedprivatedatarelease,donhauser2024privacypreservingdatareleaseleveraging,he2023algorithmicallyeffectivedifferentiallyprivate,He_2024, he2024differentiallyprivatelowdimensionalsynthetic}. The $1$-Wasserstein distance is well suited to our setting: it applies naturally to discrete probability measures, allows comparison between empirical and population distributions, and remains computationally tractable beyond the univariate case, thereby enabling empirical validation of the proposed data-release mechanism. Moreover, by means of its dual formulation, convergence in $1$-Wasserstein distance ensures the preservation of all Lipschitz statistics. For the regimes $\sigma=0$ and $\sigma<0$, we prove that PY posterior-predictive mechanism is consistent in $1$-Wasserstein distance, with respect to a ``true" data-generating distribution, and we derive explicit convergence rates. These rates make precise the privacy--utility tradeoff originally described by \cite{wasserman}: stronger privacy guarantees impose more restrictive choices of the released sample size, which in turn slow down convergence to the ``true" data-generating distribution. For $\sigma=0$, a key step in our analysis consists in bounding the expected $1$-Wasserstein distance between a posterior Dirichlet process and its posterior mean. We derive a convergence rate for this expectation, which can be interpreted as a variance term for the posterior Dirichlet process and may be of independent interest.

\subsection{Organization of the paper}
Section~\ref{background} reviews the background and notation on DP and the PY process prior. Section~\ref{our approach} introduces the PY posterior-predictive mechanism and establishes its main properties, including DP guarantees and $1$-Wasserstein consistency results. Section~\ref{num} illustrates the methodology with an application to two real datasets. Section~\ref{disc} discusses limitations of the proposed approach and outlines directions for future research. Proofs and auxiliary technical results are deferred to the appendix.

%%%%%%%%%%%%%%%%%%%%%%%%%%%%%%%%
%%%%%%%%%%%%%%%%%%%%%%%%%%%%%%%%
%%%%%%%%%%%%%%%%%%%%%%%%%%%%%%%%
%%%%%%%%%%%%%%%%%%%%%%%%%%%%%%%%

\section{Background and notation}\label{background}
\subsection{Differential privacy}

We first recall the definition of DP \citep{differentialprivacy}. We work in the global DP setting: a trusted curator observes the confidential dataset and then releases a randomized output. This should be distinguished from local DP setting, where each individual randomizes her own record before it is observed by the curator. For $n\geq 1$, let $X_1,\ldots,X_n$ be the confidential data, defined on a probability space $(\Omega,\mathcal A,\mathbb P)$ and taking values in a Polish space $\mathbb X$ with Borel $\sigma$-algebra $\mathcal X$. For $k\geq 1$, write $\mathbb X^k=\prod_{i=1}^k\mathbb X$, and let $\mathcal X^k$ be the corresponding product Borel $\sigma$-algebra. The space $\mathbb X$ is often assumed to be bounded; for example, \cite{wasserman} consider $\mathbb X=[0,1]^r$, $r\geq 1$.

In privacy-preserving statistical analysis, the goal is to construct a data-release mechanism that takes the confidential dataset $\mathbf X := (X_1,\ldots,X_n)$ as input and releases a dataset $\mathbf Z := (Z_1,\ldots,Z_m)$, with each $Z_i$ taking values in $\mathbb X$, in such a way that individual privacy is protected. In the global DP setting, the dataset $\mathbf Z$ is generated from the entire dataset $\mathbf X$. Formally, the data-release mechanism is a probability kernel $Q_n(\,\cdot\,|\,\cdot\,):\mathcal X^m\times \mathbb X^n\to[0,1]$, which gives the conditional distribution of $\mathbf Z$ given $\mathbf X$. The number $m$ of released observations is a tuning parameter and plays an important role in our results, often characterizing the privacy--utility tradeoff: increasing $m$ typically improves the utility of the released data, but weakens the privacy guarantees. 
% Differential privacy can then be formulated as a constraint on the data release mechanism. In the global privacy setting, \cite{differentialprivacy} give the following definition.

DP can be formulated as a stability requirement on the data-release mechanism: changing one record in the confidential dataset should not substantially change the distribution of the released data. Throughout the paper, we use replacement adjacency. In particular, for two datasets \(\mathbf x,\mathbf y\in\mathbb X^n\), let
\begin{displaymath}
h(\mathbf x,\mathbf y) = \left|\{i\in\{1,\ldots,n\}:x_i\neq y_i\}\right|
\end{displaymath}
be their Hamming distance. We say that \(\mathbf x\) and \(\mathbf y\) are neighbouring, or adjacent, datasets if \(h(\mathbf x,\mathbf y)=1\). In the global DP setting, \cite{differentialprivacy} introduced the following notion of DP.

\begin{definition}[$(\varepsilon,\delta)$-DP]\label{epsdelta}
Let $\varepsilon > 0$ and $\delta \in [0,1)$. We say that a data-release mechanism $Q_n$ satisfies $(\varepsilon,\delta)$-DP if for all $B \in \mathcal{X}^m$ and any $\mathbf{x},\mathbf{y} \in \mathbb{X}^n$ such that $h(\mathbf{x},\mathbf{y}) = 1$, we have
\begin{equation}\label{check}
Q_n(B \ | \ \mathbf{X} = \mathbf{x} ) \leq e^\varepsilon  Q_n( B \ | \ \mathbf{X} = \mathbf{y} ) + \delta,
\end{equation} 
\end{definition}

From Definition~\ref{epsdelta}, the parameter $\varepsilon>0$ measures the strength of the privacy guarantee: smaller values impose stronger privacy, while larger values allow the output distribution to depend more strongly on any single record. The parameter $\delta \in [0,1)$ controls the probability of privacy losses beyond $e^\varepsilon$. When Definition \ref{epsdelta} holds with $\delta=0$, the data-release mechanism is said to satisfy $\varepsilon$-DP. This notion of privacy admits a clear statistical interpretation: as shown by \cite{wasserman}, under $\varepsilon$-DP, any level-$\gamma$ test of $X_i=s$ versus $X_i=t$ has power at most $\gamma e^\varepsilon$.

Definition~\ref{epsdelta} provides a worst-case privacy guarantee: the same values of $\varepsilon$ and $\delta$ must work uniformly over all possible neighbouring datasets, which might be too stringent in some case. We therefore also use a data-dependent relaxation of $(\varepsilon,\delta)$-DP, in which the privacy requirement is imposed only at the observed dataset. Specifically, condition \eqref{check} is required to hold only for $\mathbf x$ and its neighbouring datasets, leading to the following notion, introduced by \cite{Soria_Comas_2017}.

\begin{definition}[Instance-level $(\varepsilon,\delta)$-DP]\label{epsdeltainstance}
Let $\mathbf{x} \in \mathbb{X}^n$, $\varepsilon > 0$ and $\delta \in [0,1)$. We say that a data-release mechanism $Q_n$ satisfies instance-level $(\varepsilon,\delta)$-DP if for all $B \in \mathcal{X}^m$ and any $\mathbf{y} \in \mathbb{X}^n$ such that $h(\mathbf{x},\mathbf{y}) = 1$, inequality (\ref{check}) is satisfied.
\end{definition}

For instance-level $(\varepsilon,\delta)$-DP, the values of $\varepsilon$ and $\delta$ may depend on $\mathbf x$. This relaxation still admits an interpretation in terms of hypothesis testing. Suppose that, for a given $\mathbf x\in\mathbb X^n$, the data-release mechanism satisfies (\ref{check}) for every $B\in\mathcal X^m$ and every $\mathbf y\in\mathbb X^n$ such that $h(\mathbf x,\mathbf y)=1$. Consider an attacker who observes the released data and knows the whole confidential dataset $\mathbf X=\mathbf x$, except for the value of one entry $X_i$. Then any level-$\gamma$ test of $H_0:X_i=s$ versus $H_1:X_i=t$ has power at most $e^\varepsilon\gamma+\delta$. In practice, $(\varepsilon,\delta)$-DP can be proved using the notion of $(\varepsilon,\delta)$-probabilistic DP.

\begin{definition}[$(\varepsilon,\delta)$-probabilistic DP]
\label{prob_privacy}
    Let $\varepsilon > 0$ and $\delta \in [0,1)$. We say that a data-release mechanism $Q_n$ satisfies $(\varepsilon,\delta)$-probabilistic DP  if, given $q_n(\cdot \mid \mathbf{x} )$ density of $Q_n( \cdot | \mathbf{X} = \mathbf{x})$, it holds $Q_n( A_{\mathbf{x},\mathbf{y}} \ | \ \mathbf{X} = \mathbf{x} ) > 1 - \delta$ for all $\mathbf{x},\mathbf{y}$ such that $ \ h(\mathbf{x},\mathbf{y}) = 1$, where $A_{\mathbf{x},\mathbf{y}} := \{ \mathbf{z} \in \mathbb{X}^m: \ q_n(\mathbf{z} \mid \mathbf{x})  \leq e^{\varepsilon} q_n(\mathbf{z} \mid \mathbf{y}) \}$.
 \end{definition}
 
If the Definition \eqref{prob_privacy} holds for a fixed $\mathbf x \in \mathbb X^n$ and all $\mathbf{y}$ such that $\ h(\mathbf{x},\mathbf{y}) = 1$, then the data-release mechanism $Q_n$ satisfies instance-level $(\varepsilon,\delta)$-probabilistic DP. \cite{gotz2011publishing} prove that $(\varepsilon,\delta)$-probabilistic DP implies $(\varepsilon,\delta)$-DP. Similar arguments show the same for instance-level DP guarantees.

\subsection{Pitman--Yor process prior}\label{subsec:py}

We next recall the definition of the PY process prior \citep{perman1992sizebiased,pitmanyor} and the sampling properties used throughout the paper \citep{Pitman1995ExchangeableAP}. A comprehensive account of the PY process can be found in \citet[Chapters~3 and~4]{Pitman2006CombinatorialSP} and the references therein.

Let $\mathcal P(\mathbb X)$ be the set of probability measures on $\mathbb X$, and let $H\in\mathcal P(\mathbb X)$ be a non-atomic probability measure. We write $\tilde p\sim\mathscr{PY}(\sigma,\theta,H)$ for a PY process with discount parameter $\sigma$, strength parameter $\theta$, and base measure $H$. The admissible parameter space is given either by $\sigma\in[0,1)$ and $\theta>-\sigma$, or by $\sigma<0$ and $\theta=z|\sigma|$ for some $z\in\mathbb{N}$. The definition of the PYP depends on the sign of $\sigma$:
\begin{itemize}
\item[(i)] if $\sigma\in[0,1)$, then
\begin{equation}\label{pyp01}
\tilde p=\sum_{i=1}^{\infty} W_i\delta_{Y_i},
\end{equation}
where $W_i=V_i\prod_{\ell=1}^{i-1}(1-V_\ell)$, $V_i\overset{\mathrm{ind}}{\sim}\mathrm{Beta}(1-\sigma,\theta+i\sigma)$ and $Y_i\overset{\mathrm{iid}}{\sim}H$, independently of $(V_{i})_{i\geq1}$;
\item[(ii)] if $\sigma<0$, then
\begin{equation}\label{pyp}
\tilde p=\sum_{i=1}^{z}\theta_i\delta_{Y_i},
\end{equation}
for $z\in\mathbb{N}$, where $(\theta_1,\ldots,\theta_z)\sim \mathrm{Dirichlet}(|\sigma|,\ldots,|\sigma|)$ and $Y_i\overset{\mathrm{iid}}{\sim}H$, independently of $(\theta_1,\ldots,\theta_z)$.
\end{itemize}

For every admissible value of $\sigma\in(-\infty,1)$, the PY process \eqref{pyp01}--\eqref{pyp} defines a discrete random probability measure on $\mathbb X$; that is, its realizations of the PY process are almost surely discrete probability measures. The structure of its support depends on the sign of $\sigma$. When $\sigma\in[0,1)$, the random measure has countably many atoms, with weights given by \eqref{pyp01}. The special case $\sigma=0$, denoted by $\mathscr{D}(\theta,H)$, is the Dirichlet process with strength parameter $\theta>0$ and base measure $H$ \citep{dirichlet}. When $\sigma<0$ and $\theta=z|\sigma|$ for some $z\in\mathbb N$, the PY process has exactly $z$ atoms, as in \eqref{pyp}; at the sampling level, this corresponds to a parametric Dirichlet--Multinomial model.

Because of the discreteness of $\tilde p\sim\mathscr{PY}(\sigma,\theta,H)$, a random sample $X_1,\ldots,X_n \mid \tilde p \overset{\mathrm{iid}}{\sim} \tilde p$ contains ties with positive probability. Accordingly, the sample induces a random partition of $\{1,\ldots,n\}$ whose blocks are the classes induced by the equivalence relation $i\sim j\iff X_{i}=X_{j}$. We denote by $\{X_1^*,\ldots,X_{K_n}^*\}$ the $K_n=k_{n}\leq n$ distinct values in the sample, and by $(N_1,\ldots,N_{K_n})=(n_{1},\ldots,n_{k_{n}})$, with $\sum_{i=1}^{k_{n}}n_{i}=n$, their corresponding frequencies. The random partition induced by a PY process admits a simple sequential construction through the posterior-predictive distribution of $\tilde p$. More precisely, if $X_{1},\ldots,X_{n}$ is a random sample from $\tilde p\sim\mathscr{PY}(\sigma,\theta,H)$, and $\mathbb{P}[X_{1}\in\cdot]=H(\cdot)$, then for $n\geq1$ the posterior-predictive distribution of $\tilde p$ is
\begin{equation}\label{predPy}
\mathbb P[X_{n+1}\in \cdot \mid X_{1},\ldots,X_{n}]=\frac{\theta+\sigma k_n}{\theta+n}H(\cdot)+\sum_{j=1}^{k_n}\frac{n_j-\sigma}{\theta+n}\delta_{X_j^*}(\cdot).
\end{equation}
Thus, $X_{n+1}$ is either a new value, i.e. not observed in $(X_{1},\ldots,X_{n})$, drawn from $H$, with probability
\begin{displaymath}
\frac{\theta+\sigma k_n}{\theta+n},
\end{displaymath}
or it coincides with a previously observed value, say $X_j^*$ observed in $(X_{1},\ldots,X_{n})$,  with probability
\begin{displaymath}
\frac{n_j-\sigma}{\theta+n},\qquad j=1,\ldots,k_n.
\end{displaymath}
Since $H$ is non-atomic, a draw from $H$ generates a new value with probability one. In the finite-atom case $\sigma<0$, the probability of generating a new value becomes zero when $K_n=z$, consistently with \eqref{pyp}.

The posterior-predictive distribution \eqref{predPy}, known as the Chinese restaurant process \citep{Pitman1995ExchangeableAP}, gives an interpretation of $\sigma$ and $\theta$. The strength parameter $\theta$ controls the baseline tendency to generate new values: larger values of $\theta$ increase the probability of sampling previously unseen categories, especially at small and moderate sample sizes. The discount parameter $\sigma$ controls the reinforcement mechanism and the growth of the number of distinct values. When $\sigma=0$, the predictive probability of an old value is proportional to its frequency, and the probability of a new value is $\theta/(\theta+n)$, as in the Dirichlet process. When $\sigma>0$, the mass assigned to each observed value is discounted from $n_j$ to $n_j-\sigma$, and the total discounted mass $\sigma k_n$ is transferred to the probability of generating a new value. Thus larger values of $\sigma$ favour the appearance of new categories and lead to heavier-tailed random discrete distributions. In fact, for $\sigma\in(0,1)$, the number of distinct values grows polynomially, of order $n^\sigma$, whereas under the Dirichlet process it grows logarithmically. When $\sigma<0$, instead, the model has a finite number of possible atoms.

%%%%%%%%%%%%%%%%%%%%%%%%%%%%%%%%
%%%%%%%%%%%%%%%%%%%%%%%%%%%%%%%%
%%%%%%%%%%%%%%%%%%%%%%%%%%%%%%%%
%%%%%%%%%%%%%%%%%%%%%%%%%%%%%%%%

\section{Pitman--Yor posterior-predictive mechanism}\label{our approach}

For $n\geq1$, let $\mathbf X=(X_1,\ldots,X_n)$ be the confidential dataset, and assume that $\mathbf X$ is modeled as a random sample from an unknown discrete distribution $\tilde{p}$ endowed with a PY process prior, namely
\begin{equation}\label{model1}
\begin{aligned}
X_1,\ldots,X_n \mid \tilde p &\overset{\mathrm{iid}}{\sim} \tilde p,\\
\tilde p &\sim \mathscr{PY}(\sigma,\theta,H).
\end{aligned}
\end{equation}
This model is suited to discrete data, since the PY process is a discrete random probability measure and therefore accommodates ties, repeated values, and a potentially large or unknown number of categories. Conditional on $\mathbf X$, we generate a synthetic dataset $\mathbf Z=(Z_1,\ldots,Z_m)$ by sampling sequentially from the posterior-predictive distribution induced by \eqref{model1}. The procedure relies on the one-step predictive distribution in \eqref{predPy}, iterated \(m\) times, and it is described in Algorithm~\ref{algo}.

\begin{algorithm}[t]
\caption{PYP posterior-predictive mechanism}
\label{algo}
\footnotesize\ttfamily
\begin{algorithmic}[1]
\REQUIRE Dataset $\mathbf X=(X_1,\ldots,X_n)$, synthetic sample size $m$, and PYP parameters $(\sigma,\theta,H)$.
\ENSURE Synthetic dataset $\mathbf Z=(Z_1,\ldots,Z_m)$.

\STATE Let $X_1^*,\ldots,X_k^*$ denote the distinct values in $\mathbf X$, with frequencies $n_1,\ldots,n_k$.
\STATE Set $N=n$.

\FOR{$j=1,\ldots,m$}
    \STATE Sample
    \[
    C_j\sim\mathrm{Categorical}\left(
    \frac{n_1-\sigma}{\theta+N},
    \ldots,
    \frac{n_k-\sigma}{\theta+N},
    \frac{\theta+\sigma k}{\theta+N}
    \right).
    \]
    \IF{$C_j\leq k$}
        \STATE Set $Z_j\leftarrow X_{C_j}^*$ and update $n_{C_j}\leftarrow n_{C_j}+1$.
    \ELSE
        \STATE Sample $Z_j\sim H$.
        \STATE Set $k\leftarrow k+1$, $X_k^*\leftarrow Z_j$, and $n_k\leftarrow 1$.
    \ENDIF
    \STATE Set $N\leftarrow N+1$.
\ENDFOR

\STATE Return $\mathbf Z$.
\end{algorithmic}
\end{algorithm}

We refer to Algorithm~\ref{algo} as the PY posterior-predictive mechanism. This synthetic data-generating mechanism is straightforward to implement: it only requires sequentially sampling from the PY posterior-predictive distribution \eqref{predPy}. Thus, no dedicated sampling or optimization routine is needed to obtain synthetic data, in contrast with alternative approaches, such as the exponential mechanism of \cite{mcsherry2007mechanism}, which often require problem-specific implementation.

Before studying the DP guarantees of the PY posterior-predictive mechanism, we clarify the role of the non-atomic base measure \(H\) and the meaning of categories, or distinct values. The data-release mechanism covers two common situations. In the first, the confidential data are real-valued, but the recorded observations exhibit ties because of rounding, finite measurement precision, discretization, or other preprocessing steps. In this case, the categories are simply the distinct recorded values. Since the posterior-predictive distribution may generate previously unseen values, a fixed post-processing map can be applied after sampling if the released data are required to lie on the same recorded scale as the confidential data. This map may merge some of the new categories, possibly with categories already observed in the recorded data, but it does not affect the DP guarantees, provided that it is chosen independently of the confidential data.

In the second situation, the confidential data are categorical, with recorded observations  taking values in a set of labels, such as occupations, products, locations, or species names. In this case, each observed label can be encoded by a distinct identifier in \(\mathbb X\). The non-atomicity of \(H\) then ensures that, whenever the posterior-predictive distribution selects the \(H\)-component, it generates a new identifier that almost surely does not coincide with any previously observed one. A post-processing step is then needed to map the \(\mathbb X\)-valued identifiers back to labels. Identifiers corresponding to observed categories are mapped back to their original labels, whereas new identifiers sampled from \(H\) may be assigned synthetic labels, such as \texttt{new\_category\_1}, \texttt{new\_category\_2}, and so on.

\subsection{DP guarantees}

The three regimes determined by the sign of the discount parameter $\sigma\in(-\infty,1)$, namely $\sigma\in(0,1)$,  $\sigma=0$ and $\sigma<0$, lead to markedly different DP guarantees for the PY posterior-predictive mechanism. We begin with $\sigma\in[0,1)$, for which we establish sufficient conditions for instance-level $(\varepsilon,\delta)$-DP. This weaker notion of privacy is the natural one for $\sigma\in[0,1)$, where the flexibility of the PY process prior, in terms of $\sigma$ and $\theta$, makes uniform control over all datasets too stringent. 

Fix a confidential dataset $\mathbf x=(x_1,\ldots,x_n)$ with $K_n=j$ distinct values $x_1^*,\ldots,x_j^*$ and frequencies $\mathbf n=(n_1,\ldots,n_j)$, where $n_i\geq1$, and denote by $K_m^{(n)}$ the number of new distinct values in the synthetic dataset $\mathbf Z=(Z_1,\ldots,Z_m)$ to be released. In particular, denote by $S_i:=\sum_{1\leq r\leq m} \mathbf 1_{\{x_i^*\}}(Z_r)$, for $i=1,\ldots,j$,  the number of times $x_i^*$ appears in $\mathbf{Z}$. If $Q_{\mathbf x}$ denotes the conditional distribution of $\mathbf Z$ given $\mathbf x$, then proving that the PY posterior-predictive mechanism satisfies instance-level \((\varepsilon,\delta)\)-DP requires to control the privacy loss between the output distribution \(Q_{\mathbf x}\), with density function \(q_n(\cdot\mid\mathbf x)\), and the output distribution corresponding to a neighbouring dataset \(\widetilde{\mathbf x}\), with density function \(q_n(\cdot\mid\widetilde{\mathbf x})\). Namely, we bound the likelihood ratio $q_n(\mathbf z\mid\mathbf x)/q_n(\mathbf z\mid\widetilde{\mathbf x})$ by \(e^\varepsilon\), except on a set \(\mathbf z\) having \(Q_{\mathbf x}\)-probability at most \(\delta\). This exceptional set is the \emph{bad privacy event}: on this event, $\mathbf{Z}$ is too much more likely under \(\mathbf x\) than under  \(\widetilde{\mathbf x}\), and therefore may distinguish the contribution of one individual beyond the allowed privacy level. Every neighbouring dataset is obtained by replacing one occurrence of an old value \(x_l^\ast\) in \(\mathbf x\) by a value \(b\neq x_l^\ast\). There are two cases: i) old-to-old replacement, namely \(b\) is another old value, say \(b=x_t^\ast\) with \(t\neq l\); ii) old-to-new replacement, namely \(b\) is a new value, not present among \(x_1^\ast,\ldots,x_j^\ast\).

First consider an old-to-old replacement, in which one occurrence of \(x_l^\ast\) is replaced by \(x_t^\ast\), with $t \neq l$. In the proof of
Theorem~\ref{theorem_privacy_main}, we show that the corresponding bad event can be written as
\[
\mathcal E_{l,t}^{\mathrm{old}}(\varepsilon) :=
\begin{cases}
\left\{
\dfrac{n_t-\sigma}{n_t+S_t-\sigma}
\dfrac{n_l+S_l-1-\sigma}{n_l-1-\sigma}
> e^\varepsilon
\right\}, & n_l\ge 2, \\[0.5cm]
\{S_l\ge 1\}
\cup
\left(
\{S_l=0\}
\cap
\left\{
\dfrac{\theta+(j+K_m^{(n)}-1)\sigma}{\theta+(j-1)\sigma}
\dfrac{n_t-\sigma}{n_t+S_t-\sigma}
> e^\varepsilon
\right\}
\right), & n_l=1 .
\end{cases}
\]
In the case \(n_l=1\), the replacement deletes \(x_l^\ast\) from \(\widetilde{\mathbf x}\). Since \(\mathbf Z\) is generated from \(Q_{\mathbf x}\), the value \(x_l^\ast\) may still appear in \(\mathbf Z\). Thus the event \(\{S_l\ge 1\}\) has positive \(Q_{\mathbf x}\)-probability but, because \(H\) is non-atomic, has \(Q_{\widetilde{\mathbf x}}\)-probability zero; it is therefore automatically a bad event. On the complementary event \(\{S_l=0\}\), the deleted singleton is not released, so there is no singularity. However, the two posterior-predictive distributions still differ: under \(\mathbf x\), there is one more old distinct value, while under \(\widetilde{\mathbf x}\) the count of distinct values \(x_t^\ast\) has been increased by one. The residual likelihood ratio is therefore affected by \(K_m^{(n)}\) and by \(S_t\); if many new distinct values are generated, or if few copies of \(x_t^\ast\) are released, the output is relatively more compatible with \(\mathbf x\) than with \(\widetilde{\mathbf x}\), and the privacy loss may exceed \(\varepsilon\). In the case \(n_l\ge 2\), the label \(x_l^\ast\) remains present after the replacement, so there is no singular event. The privacy loss is instead driven by how the probability mass is allocated between two affected values, say $x^{\ast}_{l}$ and $x^{\ast}_{t}$. Under \(\mathbf x\), the counts of \(x_l^\ast\) and \(x_t^\ast\) are \(n_l\) and \(n_t\), respectively, whereas under \(\widetilde{\mathbf x}\) they are \(n_l-1\) and \(n_t+1\). Hence,
\[
\dfrac{n_t-\sigma}{n_t+S_t-\sigma}
\dfrac{n_l+S_l-1-\sigma}{n_l-1-\sigma}.
\]
captures the fact that releases with many copies of \(x_l^\ast\) and few copies of \(x_t^\ast\) favour \(\mathbf x\) over \(\widetilde{\mathbf x}\).

Next consider an old-to-new replacement, in which one occurrence of \(x_l^\ast\) is replaced by a value \(b\notin\{x_1^\ast,\ldots,x_j^\ast\}\). In the proof of Theorem~\ref{theorem_privacy_main}, we show that the corresponding bad event can be written as
\[
\mathcal E_l^{\mathrm{new}}(\varepsilon) :=
\begin{cases}
\left\{
\dfrac{\theta+j\sigma}{\theta+(j+K_m^{(n)})\sigma}
\dfrac{n_l+S_l-1-\sigma}{n_l-1-\sigma}
> e^\varepsilon
\right\}, & n_l\ge 2, \\[0.5cm]
\{S_l\ge 1\}, & n_l=1,
\end{cases}
\]
whose interpretation is analogous to the one of  $\mathcal E_{l,t}^{\mathrm{old}}(\varepsilon)$. The two families of events
\(\mathcal E_{l,t}^{\mathrm{old}}(\varepsilon)\) and \(\mathcal E_l^{\mathrm{new}}(\varepsilon)\) enumerate the releases for
which a one-record replacement makes $\mathbf{Z}$ too likely under \(\mathbf x\) relative to a neighbour. Controlling their
\(Q_{\mathbf x}\)-probabilities controls the instance-level privacy loss.

\begin{theorem}[Instance-level $(\varepsilon,\delta)$-DP, \(\sigma\in[0,1)\)]\label{theorem_privacy_main}
Assume \(Q_{\mathbf x}\) is induced by the \(m\)-step posterior-predictive distribution of \(\mathscr{PY}(\sigma,\theta,H)\), where \(\sigma\in[0,1)\), \(\theta>-\sigma\), and \(H\) is a nonatomic probability measure. Let $\varepsilon>0$. With the convention that the maximum over an empty set is zero, define
\begin{equation}\label{delta_sharp}
\delta^\sharp(\varepsilon,\mathbf x) := \max\left\{ \max_{1\le l\le j} Q_{\mathbf x}\left( \mathcal E_l^{\mathrm{new}}(\varepsilon) \right),  \max_{\substack{1\le l,t\le j\\ t\neq l}} Q_{\mathbf x}\left( \mathcal E_{l,t}^{\mathrm{old}}(\varepsilon) \right) \right\}.
\end{equation}
Then the PY posterior-predictive mechanism satisfies instance-level \((\varepsilon,\delta)\)-DP at \(\mathbf x\) for every
\begin{displaymath}
\delta\in[0,1) \quad\text{such that}\quad \delta>\delta^\sharp(\varepsilon,\mathbf x).
\end{displaymath}
\end{theorem}

The quantity $\delta^\sharp$ can be approximated by Monte Carlo simulation, since \eqref{predPy} provides a direct sampling rule from the PY posterior-predictive distribution. The PY posterior-predictive mechanism can therefore be calibrated as follows. For a prescribed value of $\varepsilon$ and fixed parameters $(\sigma,\theta)$, one estimates the probabilities in \eqref{delta_sharp} by means of Monte Carlo simulation, thereby obtaining a lower bound on the admissible value of $\delta$. If this value is acceptable for the application at hand, the synthetic dataset $\mathbf{Z}$ is released according to Algorithm~\ref{algo}; otherwise, the procedure can be repeated with different choices of $(\sigma,\theta)$. If no choice of parameters yields an acceptable $\delta$, the mechanism cannot provide the desired DP guarantee for the prescribed $\varepsilon$. In the special case $\sigma=0$, the expression for $\delta^\sharp$ can be made explicit, leading to the next corollary.

\begin{corollary}[Instance-level $(\varepsilon,\delta)$-DP, $\sigma=0$]\label{cor_instance_dp_dirichlet}
Assume \(Q_{\mathbf x}\) is induced by the \(m\)-step posterior-predictive distribution of \(\mathscr{D}(\theta,H)\), where \(\theta>0\) and \(H\) is a nonatomic probability measure. Let \(\varepsilon>0\).
For \(i=1,\ldots,j\), set
\begin{equation}\label{eq:ki_def}
k(r,\varepsilon) :=
        \begin{cases}
1, & r=1, \\[0.5cm]
\left\lfloor (e^\varepsilon-1)(r-1) \right\rfloor+1, & r\ge 2.
        \end{cases}
\end{equation}
Define
\begin{displaymath}
\delta^\sharp_0(\varepsilon,\mathbf x) = \max_{1\le i\le j} \sum_{s=k_i(\varepsilon)}^m \binom ms \frac{ (n_i)_s (\theta+n-n_i)_{m-s} }{ (\theta+n)_m },
\end{displaymath}
with $k_i(\varepsilon):=k(n_i,\varepsilon)$ and the convention that the sum is zero if \(k_i(\varepsilon)>m\).
Then the PY posterior-predictive mechanism satisfies instance-level \((\varepsilon,\delta)\)-DP at \(\mathbf x\) for every
\begin{displaymath}
\delta\in[0,1) \quad\text{such that}\quad \delta>\delta^\sharp_0(\varepsilon,\mathbf x).
\end{displaymath}
\end{corollary}

The regime $\sigma=0$ is expected to provide stronger DP guarantees. The reason is that $\sigma$ governs the growth of the number of distinct values: larger values of $\sigma$ lead to more low-frequency categories, which are more vulnerable to disclosure and hence weaken privacy protection. The next theorem shows that, under the Dirichlet process prior, the proposed mechanism satisfies $(\varepsilon,\delta)$-DP.

\begin{theorem}[$(\varepsilon,\delta)$-DP, $\sigma =0$]
\label{cor_dp}
Assume \(Q_{\mathbf x}\) is induced by the \(m\)-step posterior-predictive distribution of \(\mathscr{D}(\theta,H)\), where \(\theta>0\) and \(H\) is nonatomic. Let \(\varepsilon>0\), and define
\[
\bar\delta_{n,m}(\varepsilon) := \max\left\{ \frac{m}{\theta+n+m-1}, \frac{2m}{(\theta+n)(e^\varepsilon-1)} \right\}.
\]
Then the PY posterior-predictive mechanism satisfies \((\varepsilon,\delta)\)-DP for every
\[
\delta\in[0,1) \quad\text{such that}\quad \delta>\bar\delta_{n,m}(\varepsilon).
\]
In particular, for fixed \(\theta\) and \(\varepsilon\),
\[
\bar\delta_{n,m}(\varepsilon) = O\left(\frac mn\right).
\]
\end{theorem} 
This result brings to light a structural limit of the PY posterior-predictive mechanism. The lower bound $\bar\delta_{n,m}(\varepsilon)$ in Theorem \ref{cor_dp} is defined as the maximum of two terms. The second term, $2m/((\theta+n)(e^\varepsilon-1))$, as expected in standard DP mechanisms, decays as $\varepsilon$ increases; however, the first term, the singleton contribution $m/(\theta+n+m-1)$, is independent of $\varepsilon$. Therefore, increasing $\varepsilon$ beyond a certain value does not lead to a decrease in $\bar \delta$, which plateaus at the value $m/(\theta+n+m-1)$. This means that accounting for singletons leads to disruption of the classic differential privacy trade-off between $\varepsilon$ and $\delta$: accepting a larger $\varepsilon$ does not allow to obtain $\delta$ arbitrarily close to zero. It is also impossible to arbitrarily increase the released sample size $m$ by increasing $\varepsilon$.

We finally turn to the regime $\sigma\in(-\infty,0)$. Here the model \eqref{model1} corresponds to a parametric Dirichlet--Multinomial model with at most $z=\theta/|\sigma|$ distinct atoms. This structural constraint improves the DP guarantees of the PY posterior-predictive mechanism, as formalized in the next corollary.

%Proof: $H$ uniform on $[0,1]^r$ implies $C_1 = C_2 = 1$. We take $\sigma = 0$ since we are considering a DP prior.
\begin{corollary}[Instance-level $\varepsilon$-DP, $\sigma <0$]\label{easyprop}
Assume \(Q_{\mathbf x}\) is induced by the \(m\)-step posterior-predictive distribution of $\mathscr{PY}(\sigma,\theta,H)$, where $\theta = z | \sigma |$ with $z = j$ and $\sigma < 0$. Let $\varepsilon>0$. If $m$ satisfies
\begin{equation}\label{cond_dirmult}
    m \leq (| \sigma | +n_i-1)(e^\varepsilon - 1)\text{ for all }i=1,...,j,
\end{equation}
then the PY posterior-predictive mechanism satisfies instance-level $\varepsilon$-DP.
\end{corollary}

Under the parameter specification of Corollary~\ref{easyprop}, the PY posterior-predictive mechanism reduces to the Dirichlet-Multinomial data-release mechanism of \cite{machana}; see also \cite{wasserman}; see Appendix~\ref{histograms}. Taking the infimum in \eqref{cond_dirmult} over all datasets (corresponding to the worst-case scenario of a singleton), we obtain global $\varepsilon$-DP, recovering the condition $m \leq |\sigma|(e^{\varepsilon}-1)$ from \cite{machana}. In the regime $\sigma\in(-\infty,0)$, the condition~\ref{cond_dirmult} is closely related to the condition used by \cite{wasserman}, namely
\begin{displaymath}
\alpha_i+C_i \geq \frac{m}{e^\varepsilon-1},
\end{displaymath}
where $C_i$ is the count in the $i$th histogram bin for the confidential dataset $\mathbf X$, and $\alpha_i$ is the $i$-th parameter of the Dirichlet distribution, $i=1,\ldots,j$. Overall, our DP guarantees  highlight a privacy--flexibility tradeoff for the PY posterior-predictive mechanism, with respect to the values of the discount parameter $\sigma\in(-\infty,1)$: moving from $\sigma<0$ to $\sigma=0$ and then to $\sigma\in(0,1)$ the model becomes increasingly flexible, but the associated privacy guarantees progressively weaken.

\subsection{Mechanism informativity}\label{mech_consistency}

Having established DP guarantees for the PY posterior-predictive mechanism, we now turn to the utility of the synthetic data to be released. We adopt the frequentist notion of mechanism informativity introduced in \cite{wasserman}.  Assuming that $X_1,\ldots,X_n \overset{\mathrm{iid}}{\sim} \mathfrak p_0$, for some  ``true"  data-generating distribution $\mathfrak p_0$, and that the user has access to the synthetic dataset $\mathbf Z=(Z_{1},\ldots,Z_{m})$ but not to the original confidential data $\mathbf{X}=(X_{1},\ldots,X_{n})$, we explore how inference based on $\mathbf Z$ is informative about aspects of  $\mathfrak p_0$. This distinction is important in practice: a mechanism that releases data with negligible statistical information may provide strong DP guarantees, but it is of little inferential value. We therefore ask whether the empirical distribution of synthetic data converges to the ``true" data-generating distribution as the sample size $m$ grows. Under the frequentist setup, the marginal distribution of $\mathbf Z$ is then the probability measure
\begin{equation}\label{marginal_nu}
\nu_{n,m}(A):=\int_{\mathbb X^n}Q_n(A\mid \mathbf X=\mathbf x)\,\mathfrak p_0(dx_1)\cdots \mathfrak p_0(dx_n),
\end{equation}
for all $A\in\mathcal X^m$. We allow the number of synthetic  observations to depend on the size $n$ of the confidential dataset, writing $m=m(n)$. Accordingly, asymptotic statements as $n\to\infty$ implicitly involve the corresponding sequence of sample sizes $m(n)$. Once a relationship between $m$ and $n$ has been specified, we write $\nu_n$ in place of $\nu_{n,m}$ for the marginal distribution in \eqref{marginal_nu}. Let
\begin{displaymath}
e_m^{(\mathbf Z)}:=\frac1m\sum_{i=1}^m \delta_{Z_i}
\end{displaymath}
denote the empirical measure of the synthetic data $\mathbf Z$. \cite{wasserman} give the following definition of mechanism informativity, which provides a natural measure of the quality of the mechanism.

\begin{definition}[Mechanism informativity]\label{inf}
We say that a data-release mechanism $Q_n$ is informative, or consistent, with respect to the distance $d$ if $d(e_m^{(\mathbf{Z})},\mathfrak{p}_0) \rightarrow 0$ in $\nu_{n}$-probability as $n \to +\infty$. $Q_n$ is $\varepsilon_n$-informative if $\mathbb{E}_{\nu_{n}}[d( e_m^{(\mathbf{Z})},\mathfrak{p}_0)] = O(\varepsilon_n)$, where $(\varepsilon_n)_{n \geq 1}$ is a sequence of positive numbers such that $\varepsilon_n \rightarrow 0$ as $n \rightarrow \infty$.
\end{definition}

Definition~\ref{inf} is formulated for a generic distance between probability measures. In this work, we specialize it to the $1$-Wasserstein distance. We recall the general definition of the $p$-Wasserstein distance; see, e.g., \citet{Ambrosio2013} and \citet{villani2008optimal}. Let $(\mathbb X,d_{\mathbb X})$ be a complete separable metric space. For $p\geq1$ and $\gamma_1,\gamma_2\in\mathcal P_p(\mathbb X)$, the $p$-Wasserstein distance is defined as
\begin{equation}
\label{wasserstein}
\mathcal W_p^{\mathcal P_p(\mathbb X)}(\gamma_1,\gamma_2):=\inf_{\xi\in\Pi(\gamma_1,\gamma_2)}\left(\int_{\mathbb X^2} d_{\mathbb X}(x,y)^p\,\xi(dx,dy)\right)^{1/p},
\end{equation}
where
\[
\mathcal P_p(\mathbb X):=\left\{\gamma\in\mathcal P(\mathbb X):\int_{\mathbb X} d_{\mathbb X}(x,x_0)^p\,\gamma(dx)<\infty\text{ for some } x_0\in\mathbb X\right\}.
\]
Here, $\Pi(\gamma_1,\gamma_2)$ denotes the set of couplings of the probability measures $\gamma_1$ and $\gamma_2$, that is, the set of probability measures on $(\mathbb X^2,\mathcal X^2)$ with marginals $\gamma_1$ and $\gamma_2$. Throughout the paper we assume that $\mathbb X$ is totally bounded, so that $\mathcal P_p(\mathbb X)=\mathcal P(\mathbb X)$ holds for every $p\geq1$. When the underlying space is clear from the context, we omit the superscript $\mathcal P_p(\mathbb X)$ in $\mathcal W_p^{\mathcal P_p(\mathbb X)}$ and simply write $\mathcal W_p$.

For $\mathbb X\subseteq\mathbb R^d$, rates for $\varepsilon_{n,1}(\mathbb X,\mathfrak p_0):=\mathbb E_{\mathfrak p_0^n}[\mathcal W_1(e_m^{(\mathbf X)},\mathfrak p_0)]$ are already available from \cite{fournier2015rate}; see Appendix~\ref{empirical}. In our setting, the additional difficulty in verifying Definition~\ref{inf} is that the empirical measure is computed from the synthetic data $\mathbf Z$, rather than from the original confidential data $\mathbf X$. It is also important to ensure that the DP guarantees of a data-release mechanism are preserved as $n\to\infty$. This imposes constraints on the growth of $m=m(n)$ relative to $n$, and these constraints play a crucial role in the mechanism informativity.  We therefore establish $1$-Wasserstein consistency for the two regimes in which the PY posterior-predictive mechanism provides satisfactory DP guarantees, namely $\sigma\in(-\infty,0)$ and $\sigma=0$. Although the PY posterior-predictive mechanism for $\sigma\in(-\infty,0)$ has already appeared in the literature \citep{machana}, its consistency in the sense of Definition~\ref{inf} has not been established.

\begin{theorem}[$\sigma<0$]
\label{consistenza multinomiale}
    Let $\mathbb{X} \subseteq \mathbb{R}^d$. If $Z_1,...,Z_m$ are synthetic from  the PY posterior-predictive mechanism with $\sigma<0$ and $\alpha_i \geq 2$ for all $i =1,...,z$ and $z \geq 2$, then
    \begin{equation*}
        \mathbb{E}_{\nu_{n,m}}\left [\mathcal{W}_1(e_m^{(\mathbf{Z})},\mathfrak{p}_0) \right ] = O(n^{-1/2}) + \begin{cases}  O(m^{-1/2}) \quad & \text{if } d < 2, \\  O(m^{-1/2}\log(1+m)) \quad & \text{if } d = 2, \\
         O(m^{-1/d}) \quad & \text{if } d > 2. \end{cases} 
    \end{equation*}
\end{theorem}

Now consider the setting of Theorem~\ref{consistenza multinomiale} and assume further that $m\asymp n$. It turns out that the condition
\begin{equation*}
\label{priv_cond}
    \alpha_i + C_i \geq m / (e^\varepsilon - 1) \text{ for all }i=1,...,z
\end{equation*} 
required in the Dirichlet--Multinomial regime to satisfy instance-level $\varepsilon$-DP is met asymptotically. Indeed, as $n\to\infty$, the bin counts $C_i$, $i=1,\ldots,z$, diverge. To see this, assume that $\mathfrak p_0=\mathscr T(\theta_0)$, where $\mathscr T:\Theta\to\mathcal P(\mathbb X)$ is a bijective parametrization map and $\theta_0=(\theta_{0,1},\ldots,\theta_{0,z})$. Then $(C_1,\ldots,C_z)\sim\mathrm{Multinomial}(n,\theta_0)$, so that each $C_i$ is marginally distributed as a binomial random variable with parameters $(n,\theta_{0,i})$. Therefore, by Chebyshev's inequality,
\begin{displaymath}
\mathbb P\left(\left|C_i/n-\theta_{0,i}\right|>\varepsilon\right)\leq\frac{\theta_{0,i}(1-\theta_{0,i})}{n\varepsilon^2}\to 0.
\end{displaymath}
Thus, provided $\theta_{0,i}>0$, each $C_i$ grows linearly with $n$, at least in probability. Therefore, when $m\asymp n$, Theorem~\ref{consistenza multinomiale} yields
\begin{displaymath}
\mathbb E_{\nu_n}\left[\mathcal W_1\left(e_m^{(\mathbf Z)},\mathfrak p_0\right)\right]=
\begin{cases}
O(n^{-1/2}), & d<2,\\
O(n^{-1/2}\log(1+n)), & d=2,\\
O(n^{-1/d}), & d>2.
\end{cases}
\end{displaymath}
This result shows that, under the Dirichlet--Multinomial regime, the PY posterior-predictive mechanism is informative with the same rate as that obtained by \cite{fournier2015rate}. In this regime, replacing the confidential data by synthetic data therefore entails no loss in the consistency rate.

We next extend Theorem~\ref{consistenza multinomiale} from the regime $\sigma\in(-\infty,0)$ to the regime $\sigma=0$. For $\sigma=0$, the Dirichlet process prior is no longer supported on a fixed finite set of atoms, but allows for countably many possible categories and for the appearance of new values in the synthetic sample. This additional flexibility requires a separate argument to establish $1$-Wasserstein consistency.

\begin{theorem}[$\sigma=0$]
\label{consistencydir}
Let $\mathbb{X}\subseteq \mathbb{R}^d$ be bounded. If $Z_{1},...,Z_{m}$ are synthetic from the PY posterior-predictive mechanism with $\sigma=0$ then
\begin{equation*}
\mathbb{E}_{\nu _{n,m}}\left( \mathcal{W}_{1}\left( e_{m}^{\left( \mathbf{Z}\right) },\mathfrak{p}_{0}\right)
\right) = \varepsilon _{m,1}\left( \mathbb{X}\right) + \begin{cases} O(n^{-1/2}) \quad & \text{if } d < 2,  \\
    O(n^{-1/2}\log(1+n)) \quad & \text{if } d = 2, \\
    O(n^{-1/d})  \quad & \text{if } d > 2 .
    \end{cases} 
\end{equation*}
where $\varepsilon _{m,1}$ is the rate of convergence for the mean of the $1$-Wasserstein distance between $e_{m}^{\left( 
\mathbf{Z}\right) }$ \ and $P|X_1,...,X_n\sim \pi _{n}(\cdot|X_1,...,X_n)$, with $\pi _{n}$ being the distribution of a posterior Dirichlet process.
\end{theorem}

We now consider the problem of specifying the asymptotic relation between the private sample size $n$ and the synthetic sample size $m=m(n)$. Since the term $\tilde\delta$ in Theorem~\ref{cor_dp} is of order $O(m/n)$, the asymptotic validity of the assumptions ensuring $(\varepsilon,\delta)$-DP imposes restrictions on the growth of $m$ relative to $n$. One possibility is to fix a target privacy level $\delta\in(0,1)$ and let $m\asymp n$, with a sufficiently small proportionality constant. Indeed, if $m=\lfloor \alpha n\rfloor$, then it holds that
\begin{displaymath}
\limsup_{n\to\infty} \bar\delta_{n,m}(\varepsilon)\leq\max\left\{\frac{\alpha}{1+\alpha},\frac{2\alpha}{e^\varepsilon-1}\right\}.
\end{displaymath}
Therefore, any choice of $\alpha<\min\{\delta/(1-\delta),\delta(e^\varepsilon-1)/2\}$ ensures $(\varepsilon,\delta)$-DP for all sufficiently large $n$.

\begin{corollary}[$\sigma=0$]
\label{consistencydir_epsdeltadp}
Let $\mathbb{X}\subseteq \mathbb{R}^d$ be bounded. If $Z_{1},...,Z_{m}$ are synthetic data from the PY posterior-predictive mechanism with $\sigma=0$ and $m \asymp n$, then
\begin{equation}\label{rate2}
\begin{aligned}
    \mathbb{E}_{\nu _{n}}\left( \mathcal{W}_{1}\left( e_{m}^{\left( \mathbf{Z}\right) },\mathfrak{p}_{0}\right)
\right)  = \begin{cases} O(n^{-1/2}) \quad & \text{if } d < 2,  \\
    O(n^{-1/2}\log(1+n)) \quad & \text{if } d = 2, \\
    O(n^{-1/d})  \quad & \text{if } d > 2 .
    \end{cases}
\end{aligned}
\end{equation}
which is minimax optimal. 
\end{corollary}

If we make the stronger requirement $\bar{\delta}_{n,m}\rightarrow0$ as $n \to +\infty$, so that $\varepsilon$-DP holds asymptotically, then we need to assume $m=o(n)$. In particular if $m \asymp \frac{n}{\log n}$, then we have the following corollary
\begin{corollary}[$\sigma=0$]
\label{consistencydir_epsdp}
Let $\mathbb{X}\subseteq \mathbb{R}^d$ be bounded. If $Z_{1},...,Z_{m}$ are synthetic data from the PY posterior-predictive mechanism with $\sigma=0$ and $m \asymp \frac{n}{\log n}$, then
\begin{equation}\label{rate1}
\begin{aligned}
    \mathbb{E}_{\nu _{n}}\left( \mathcal{W}_{1}\left( e_{m}^{\left( \mathbf{Z}\right) },\mathfrak{p}_{0}\right)
\right)  = \begin{cases} O(n^{-1/2}\sqrt{\log n}) \quad & \text{if } d < 2,  \\
    O(n^{-1/2}(\log n)^{3/2}) \quad & \text{if } d = 2, \\
    O(n^{-1/d}(\log n)^{1/d})  \quad & \text{if } d > 2 .
    \end{cases}
\end{aligned}
\end{equation}
\end{corollary}

In the regime $\sigma=0$, the PY posterior-predictive mechanism is informative with respect to the $1$-Wasserstein metric. More precisely, it is informative with rate of convergence \eqref{rate2} when regarded as an $(\varepsilon,\delta)$-DP mechanism, and with rate of convergence \eqref{rate1} when regarded as an asymptotically $\varepsilon$-DP mechanism. In the former case, the weaker DP requirement leads to a faster convergence rate, which is minimax optimal. Thus, under $(\varepsilon,\delta)$-DP, replacing the confidential data by synthetic data entails no loss in the rate of convergence of the empirical measure to the data-generating distribution. This behaviour illustrates the general phenomenon, emphasized by \cite{wasserman}, that DP constraints may slow down convergence rates. At the same time, our results address their question on whether minimax rates can be obtained under DP constraints, showing that they can be achieved by assuming the weaker $(\varepsilon,\delta)$-DP.

Theorems~\ref{consistenza multinomiale} and~\ref{consistencydir} also imply corresponding bounds for the expected distance between the synthetic empirical measure and the empirical measure of the confidential data. Indeed, by the
triangle inequality,
\[
\mathbb E\left[ \mathcal W_1\left(e_m^{(\mathbf Z)},e_n^{(\mathbf X)}\right) \right]
\leq \mathbb E\left[\mathcal W_1\left(e_m^{(\mathbf Z)},\mathfrak p_0\right) \right] +
\mathbb E\left[\mathcal W_1\left(e_n^{(\mathbf X)},\mathfrak p_0\right) \right].
\]
This empirical-to-empirical target is considered by \citet{boedihardjo2024privatemeasuresrandomwalks}. Their mechanism satisfies
pure \(\varepsilon\)-DP for every $n$ and obtains a Wasserstein error of order \(n^{-1/d}\), up to logarithmic factors.
By contrast, in Corollary~\ref{consistencydir_epsdp} we establish the same rates,  but under the weaker finite-sample guarantee \((\varepsilon,\delta_n)\)-DP with \(\delta_n\to0\).  The improvement in \citet{boedihardjo2024privatemeasuresrandomwalks} comes at the price of a more involved construction based on superregular random walks, whereas the PYP posterior-predictive mechanism retains the the structure of the Bayesian nonparametric model.

The proof of Theorem~\ref{consistencydir} relies on a concentration bound for the posterior distribution of a Dirichlet process, which may be of independent interest. The key quantity is the expected $1$-Wasserstein distance between a posterior Dirichlet process  and its corresponding posterior mean measure. This term can be interpreted as a Wasserstein analogue of a posterior variance of the Dirichlet process, measuring how tightly the posterior random measure concentrates around its mean.

\begin{lemma}\label{concDP}
    Let $\tilde{p}\sim\mathscr{D}\left( \theta ,H\right)$ and let $X_1,...,X_n$ be a random sample from $\tilde{p}$, for $n\geq1$, such that $\tilde{p}|X_1,...,X_n\sim \pi _{n}\left( \cdot |X_1,...,X_n\right) =\mathscr{D}%
\left( \theta +n,H_n\right)$, where $H_n=\frac{\theta}{\theta+n}H+\frac{n}{\theta+n}\sum_{i=1}^n\delta_{X_i}$. If $X_1,...,X_n \iid \mathfrak{p}_0$, then
\begin{equation}\label{dir_rate}
\mathbb{E}_{\mathfrak{p}%
_{0}^{n}} \left( \mathbb{E}_{\tilde{p}|X_1,..,X_n}\left( \mathcal{W}_{1}\left(
P,H_{n}\right) \right) \right) = \begin{cases} O(n^{-1/2}) \quad & \text{if } d < 2,  \\
    O(n^{-1/2}\log(1+n)) \quad & \text{if } d = 2, \\
    O(n^{-1/d})  \quad & \text{if } d > 2. 
    \end{cases} 
\end{equation}
\end{lemma}

Interestingly, Lemma \ref{concDP} implies, by means of triangle inequality, that the same rate is achieved by
\begin{equation}\label{pcr}
    \mathbb{E}_{\mathfrak{p}_{0}^{n}}\left( \mathcal{W}_{1}^{\mathcal{P}(\mathcal{P}\left( 
\mathbb{X})\right) }\left( \pi _{n}\left( \cdot |X_1,...,X_n\right)
,\delta _{\mathfrak{p}_{0}}\right) \right) 
\end{equation}
which is a posterior contraction rate at $\mathfrak{p}_0$ for the Dirichlet process prior. Rates of the form in (\ref{pcr}) have been studied by \cite{camerlenghi2026wasserstein}, though their techniques yield a slower rate than \eqref{dir_rate}.

%%%%%%%%%%%%%%%%%%%%%%%%%%%%%%%%
%%%%%%%%%%%%%%%%%%%%%%%%%%%%%%%%
%%%%%%%%%%%%%%%%%%%%%%%%%%%%%%%%
%%%%%%%%%%%%%%%%%%%%%%%%%%%%%%%%

\section{Numerical illustrations}\label{num}

We report three numerical experiments. The first one illustrates the main computational advantage of posterior-predictive release for downstream Bayesian inference. The second one validates the Wasserstein informativity rates of Theorem \ref{consistencydir} in a controlled simulation. The third one compares distributional utility on a large real dataset from the American Community Survey.

\subsection{Downstream Bayesian inference}\label{downstream}

Recall that under the Dirichlet-process posterior-predictive mechanism, the released data can be analyzed with the ordinary Dirichlet-process posterior. Indeed, if \(Z_1,\ldots,Z_m\) are the released observations, then the downstream posterior is available in closed form:
\[
\tilde p \mid Z_1,\ldots,Z_m=z_1,\ldots,z_m \sim \mathscr D\left(\theta+m,\frac{\theta H+\sum_{\ell=1}^{m}\delta_{z_\ell}}{\theta+m}\right).
\]
Thus posterior summaries can be computed by direct posterior simulation from a standard Dirichlet-process posterior.

This is in sharp contrast with generic synthetic-data mechanisms. Consider, for instance, the perturbed histogram mechanism reviewed in Appendix \ref{histograms}. Let \(B_1,\ldots,B_k\) be the histogram partition, let \(C_j=\sum_{i=1}^{n}\indicator[X_i\in B_j]\), and let \(\tilde f\) be the perturbed histogram density obtained from the noisy counts. If \(Z_1,\ldots,Z_m\) are then sampled from \(\tilde f\), the correct Bayesian posterior is not the ordinary posterior given \(Z_1,\ldots,Z_m\). Rather, it is the mechanism-aware posterior associated with the hierarchical model
\[
\tilde p \sim \mathscr D(\theta,H),\qquad X_1,\ldots,X_n \mid \tilde p \iid \tilde p,\qquad C_j=\sum_{i=1}^{n}\indicator[X_i\in B_j],\qquad Z_1,\ldots,Z_m \mid \tilde f \iid \tilde f.
\]
Consequently, posterior computation requires augmenting the Markov chain with the unobserved confidential sample \(X_{1:n}\), and then updating the Dirichlet process conditionally on the imputed confidential data:
\[
\tilde p \mid X_1,\ldots,X_n \sim \mathscr D\left(\theta+n,\frac{\theta H+\sum_{i=1}^{n}\delta_{X_i}}{\theta+n}\right).
\]
The latent state therefore scales with the confidential sample size \(n\), not with the released sample size \(m\). This makes the exact mechanism-aware posterior computation rapidly infeasible. In our reference implementation, for \(n=10^4\) and \(m=10^2\), obtaining \(1000\) posterior samples from the full \(X_{1:n}\)-augmented MCMC takes about six minutes, making it unfeasible in the intended large-scale regime, where \(n\) may be \(10^6\) or larger.

For this reason, in the repeated simulation below we do not use the full \(X_{1:n}\)-augmented posterior for the histogram sample-release mechanisms. Instead, we use a computational fallback that is deliberately favorable to the histogram mechanisms: we replace the synthetic sample from the histogram by a direct release of the histogram output. This reduces the latent state from the full confidential sample \(X_{1:n}\) to the confidential bin counts. For the perturbed histogram, a direct release of noisy counts is a valid differentially private release when the noise is calibrated to the sensitivity of the count vector. For the smoothed histogram, however, directly releasing the smoothed histogram is not differentially private, except in the degenerate case of complete smoothing, which has no utility. The direct smoothed histogram should therefore be understood only as a non-private computational diagnostic.

We run the fallback simulation on \(n=2000\) confidential observations generated from a discrete distribution on a regular grid in \([0,1]\). The grid probabilities are generated from a mixture of beta-shaped weights, as described in Appendix \ref{details_num}. We compare the downstream posterior summaries obtained from three workflows: the Dirichlet-process posterior-predictive mechanism, the direct perturbed histogram, and the direct smoothed histogram. Posterior samples are used to estimate the mean, the quartiles, and the right-tail probability \(\tilde p((0.75,1])\). We report empirical bias, root mean square error, average credible-interval length, frequentist coverage, wall-clock time, and effective sample size per second. Complete results are given in Appendix \ref{details_num}. The statistical accuracy of the three methods is comparable in this simplified experiment, but the computational workflows are not: the proposed posterior-predictive mechanism requires only the standard Dirichlet-process posterior update, whereas the histogram-based workflows require mechanism-aware inference even in the reduced direct-histogram formulation.

\subsection{Convergence analysis for a simulated dataset}\label{conv}

We now validate empirically the Wasserstein informativity rates proved in Theorem \ref{consistencydir}. The confidential data are generated from an explicitly discrete distribution on \([0,1]\), constructed as follows. Let \(G\) have a geometric distribution on \(\{1,2,\ldots\}\) with success probability \(p=0.05\). For \(k\geq 1\), set \(r(k)=\lfloor \log_2 k \rfloor+1\), and define
\[
    T(k)=\frac{2(k-2^{r(k)-1})+1}{2^{r(k)}}.
\]
Thus \(T(1)=1/2\), \(T(2)=1/4\), \(T(3)=3/4\), \(T(4)=1/8\), and so on. We generate \(N=1000000\) iid observations \(G_1,\ldots,G_N\), and set \(X_i=T(G_i)\). The resulting population sample takes values in the dyadic points of \([0,1]\), contains ties by construction, and has a large but sparse support.

The dataset of size \(N\) is treated as the confidential population. For each value of \(n\) in a grid between \(10000\) and \(N\), we draw a random subsample \(X_1,\ldots,X_n\) and generate a synthetic dataset \(Z_1,\ldots,Z_m\) using the Dirichlet-process posterior-predictive mechanism with base measure \(H=\operatorname{Unif}[0,1]\) and strength parameter \(\theta\). Throughout this section we consider the Pitman--Yor posterior-predictive mechanism in the special case \(\sigma=0\), and refer to it as the \textit{Dirichlet-process posterior-predictive mechanism}.

The synthetic sample size \(m\) is chosen according to the privacy regime under consideration. For \((\varepsilon,\delta)\)-differential privacy, we set \(\varepsilon=2\) and consider \(\delta\in\{10^{-2},10^{-3},10^{-4}\}\); for each \(n\), \(m\) is chosen as the largest value satisfying the finite-sample bound of Theorem \ref{cor_dp}. We also consider the asymptotic \(\varepsilon\)-differential privacy regime of Corollary \ref{rate1}, using a sequence \(m=m(n)\) with \(m/n\to 0\). For each configuration, we compute the \(1\)-Wasserstein distance between the synthetic empirical measure and the empirical measure of the confidential subsample,
\[
\mathcal W_1\left(e_m^{(\mathbf Z)},e_n^{(\mathbf X)}\right),
\]
and average the result over \(100\) independent runs of the synthetic-data generation process.

Figure \ref{w1_fig} displays the results for \(\theta\in\{1,10,100\}\). In the three \((\varepsilon,\delta)\)-differentially private scenarios, the decay is close to the \(n^{-1/2}\) benchmark on the log-log scale, in agreement with Corollary \ref{consistencydir_epsdeltadp}. In the asymptotic \(\varepsilon\)-differential privacy regime, the convergence is slower, consistently with Corollary \ref{rate1}. Although the asymptotic \(\varepsilon\)-DP curve can lie below some of the finite-\(\delta\) curves for the sample sizes considered here, this does not contradict the theory: pure \(\varepsilon\)-DP is only obtained asymptotically in this regime, and the finite-sample value of the corresponding \(\delta\) is larger than \(10^{-2}\).

\begin{figure}[htbp]
    \centering
    \includegraphics[width=0.33\textwidth]{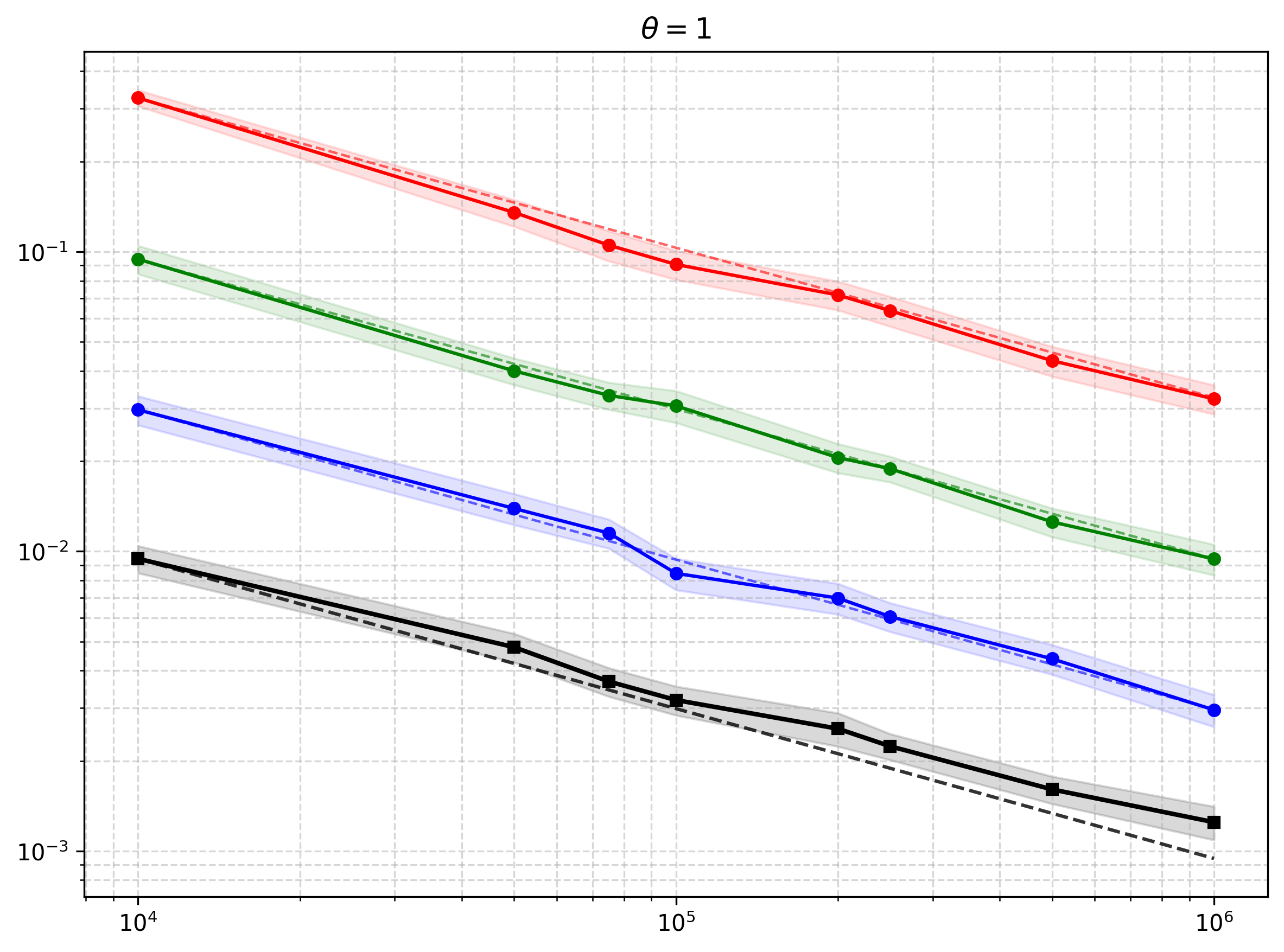}%
    \includegraphics[width=0.33\textwidth]{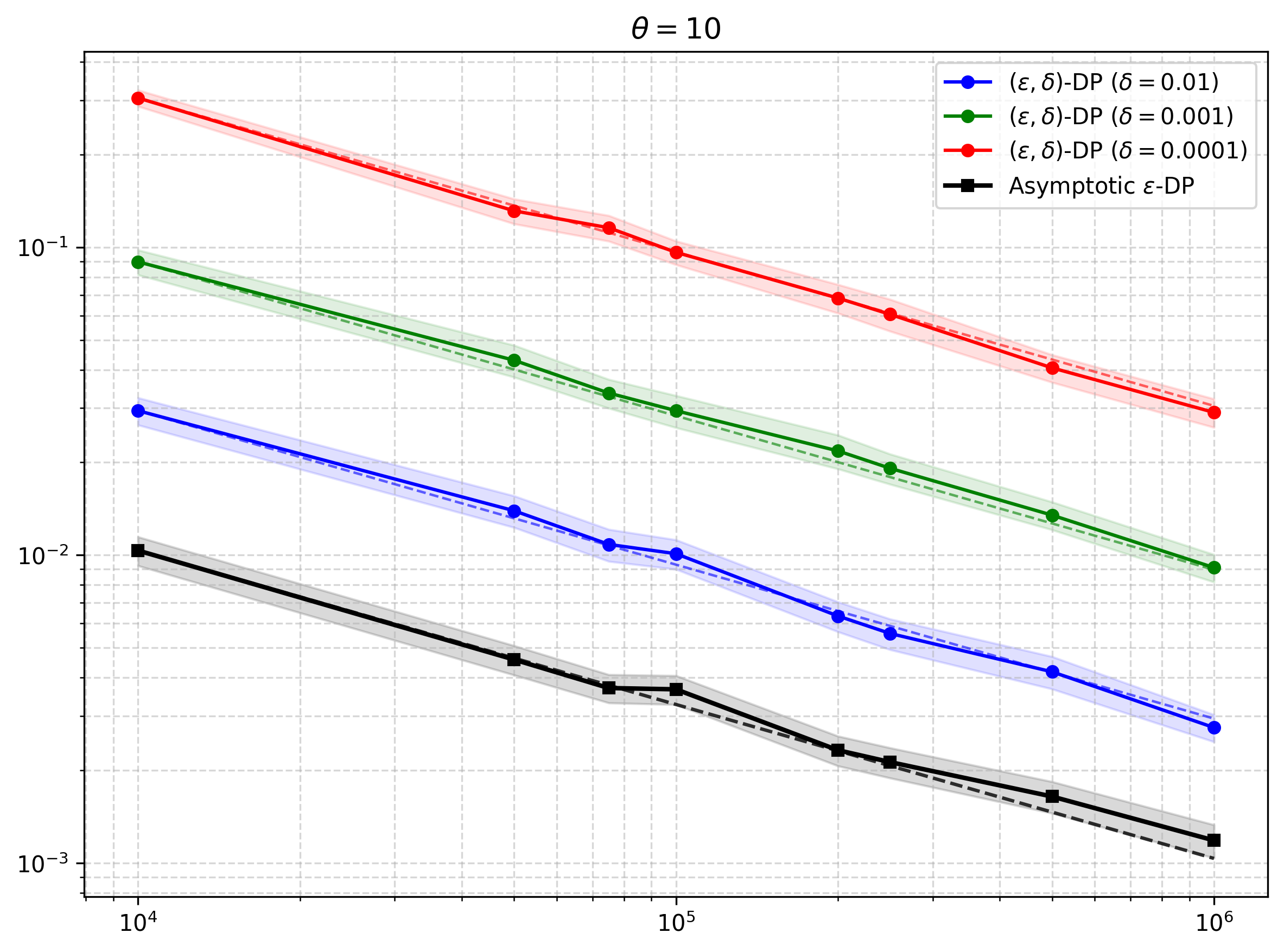}%
    \includegraphics[width=0.33\textwidth]{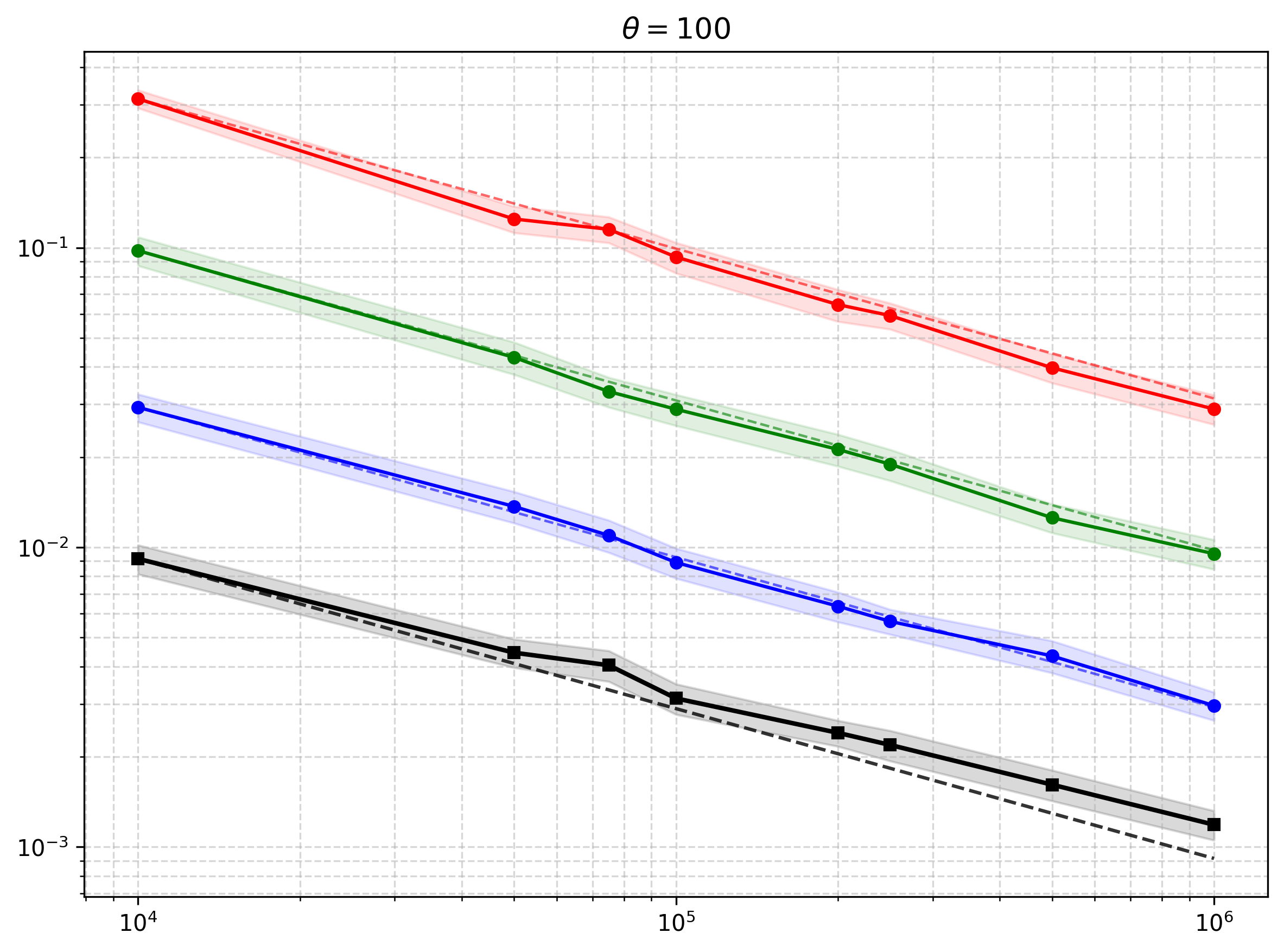}
    \caption{Average \(\mathcal W_1\left(e_m^{(\mathbf Z)},e_n^{(\mathbf X)}\right)\) as a function of \(n\), with a \(99\%\) confidence band, on a log-log scale, for \(\theta=1,10,100\). Solid curves correspond to computed distances; dashed lines have slope \(-0.5\).}
    \label{w1_fig}
\end{figure}

\subsection{Real data: California census dataset}\label{cal_data}

We finally assess distributional utility on a large real dataset. We consider the \(2024\) \(5\)-Year Public Use Microdata Sample from the American Community Survey, conducted by the U.S. Census Bureau \citep{acs2021california}, and focus on the PINCP variable, which records personal annual income. Observations with missing or non-positive income are removed. The remaining values are transformed logarithmically and then rescaled to \([0,1]\). The resulting dataset contains \(11918162\) observations and \(32393\) distinct recorded values, of which \(7689\) are singletons.

We compare the Dirichlet-process posterior-predictive mechanism with the perturbed histogram and smoothed histogram mechanisms reviewed in Appendix \ref{histograms}. The confidential data are modeled with a Dirichlet process prior with base measure \(H=\operatorname{Unif}[0,1]\) and \(\theta=1\). We set \(\varepsilon=2\) and \(\delta=10^{-5}\). Under Theorem \ref{cor_dp}, the largest admissible released sample size is \(m=119\). This is a stringent privacy regime: the number of released observations is tiny compared with the confidential sample size.

For the histogram mechanisms, we use the same number \(m=119\) of released synthetic observations and \(300\) histogram bins. To match the \((\varepsilon,\delta)\)-privacy target, the perturbed histogram uses Gaussian perturbations of the histogram counts calibrated to the sensitivity of the count vector. For the smoothed histogram, we compose the smoothing step with an \((\varepsilon,\delta)\)-differentially private Gaussian perturbation of the counts. We then generate synthetic datasets from the resulting private histograms.

We evaluate distributional utility in three ways. First, we compare standard summary statistics of the confidential and synthetic datasets: mean, standard deviation, and quartiles. Second, we compare kernel density estimates computed from the confidential and synthetic samples. Third, over \(100\) independent runs of each mechanism, we compute the \(1\)-Wasserstein distance between the synthetic empirical measure and the empirical measure of the confidential data.

\renewcommand{\arraystretch}{1.3}
\begin{table}[htbp]
\centering
\footnotesize
\begin{tabular}{|p{10em} c c c c c|}
\hline
 & \textbf{Mean} & \textbf{SD} & \textbf{Q1 (25\%)} & \textbf{Median} & \textbf{Q3 (75\%)} \\
\hline \hline
\textbf{Confidential} & 0.698 & 0.093 & 0.651 & 0.711 & 0.759 \\
\hline
\textbf{Perturbed} & 0.697 & 0.096 & 0.652 & 0.710 & 0.757 \\
\textbf{Smoothed} & 0.693 & 0.087 & 0.639 & 0.704 & 0.748 \\
\textbf{Dirichlet-process posterior-predictive} & 0.691 & 0.091 & 0.632 & 0.707 & 0.759 \\
\hline
\end{tabular}
\caption{Summary statistics of confidential data and synthetic data generated with three different mechanisms.}
\label{tab:summary_statistics}
\end{table}

Table \ref{tab:summary_statistics} shows that all three mechanisms preserve the main one-dimensional summaries reasonably well, despite the small value of \(m\). The large confidential sample size makes the histogram mechanisms particularly favorable in this example: most bin counts are large, so the perturbation added to the counts has a limited effect, and the smoothed histogram does not need to over-smooth the distribution. The Dirichlet-process posterior-predictive mechanism nevertheless remains competitive, without requiring a fixed histogram partition.

Figures \ref{fig:KDE} and \ref{fig:wass_boxplots} give a similar picture. The kernel density estimates obtained from the three synthetic samples are close to the estimate based on the confidential data, and the average \(1\)-Wasserstein distance is below \(0.0125\) for all mechanisms. Thus, on this large real dataset, the proposed mechanism achieves distributional utility comparable to the histogram-based methods under the same released sample size.

\begin{figure}[htbp]
    \centering
    \begin{minipage}[c]{0.48\textwidth}
        \centering
        \includegraphics[width=\textwidth]{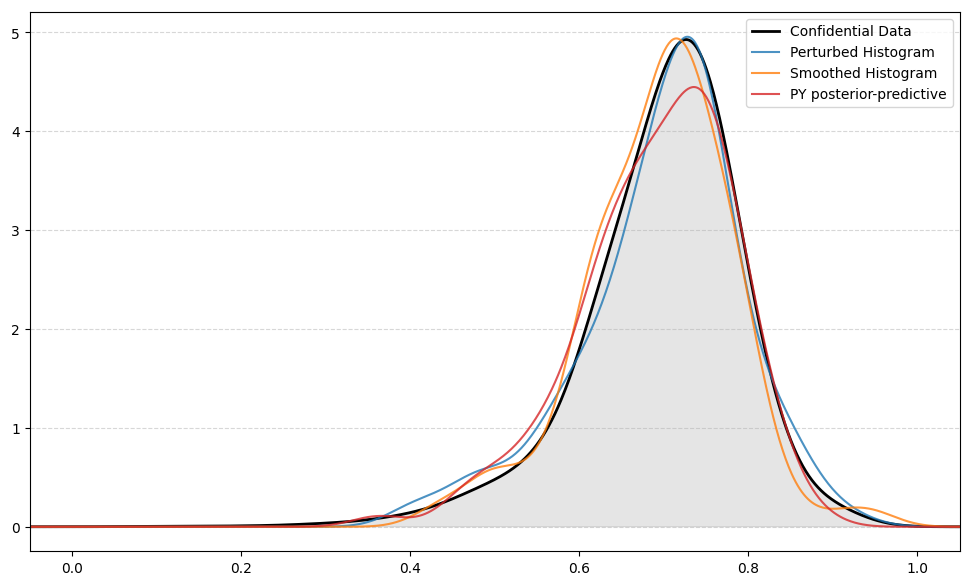}
        \caption{Kernel density estimates based on confidential and synthetic datasets.}
        \label{fig:KDE}
    \end{minipage}\hspace{0.02\textwidth}%
    \begin{minipage}[c]{0.48\textwidth}
        \centering
         \includegraphics[width=\textwidth]{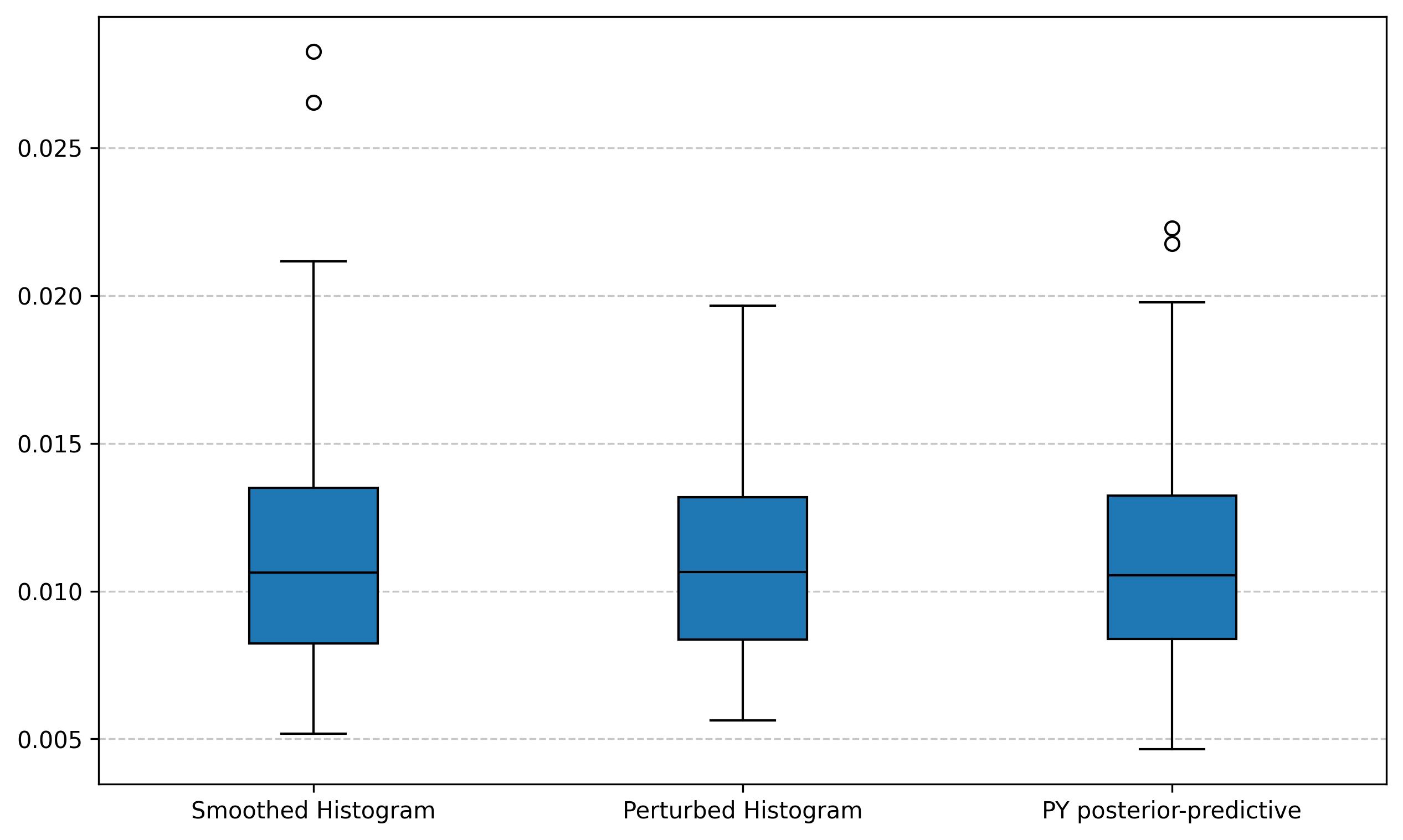}
        \caption{Boxplots of the \(1\)-Wasserstein distance between synthetic and confidential empirical measures, over \(100\) independent runs.}
        \label{fig:wass_boxplots}
    \end{minipage}
\end{figure}

\section{Discussion}\label{disc}

We have proposed a Bayesian nonparametric framework for private synthetic data generation tailored to discrete data. The mechanism is model-aligned: synthetic observations are drawn directly from the posterior-predictive distribution of the Bayesian nonparametric model. Hence, under the model, confidential and synthetic data are jointly exchangeable, allowing standard Bayesian inference to be carried out from the released data. This contrasts with generic externally specified privatization mechanisms, for which likelihood-based or Bayesian inference may require integrating over the unobserved confidential sample and can become computationally intractable. The proposed mechanism is also simple to implement, since sampling from the PY posterior-predictive distribution is straightforward. At the same time, because the PY process is naturally suited to discrete data with ties and possibly many categories, extensions to data without ties or to mixed discrete-continuous settings remain important directions for future work.

The DP guarantees of the PY posterior-predictive mechanism depend on the discount parameter $\sigma\in(-\infty,1)$, which controls the prevalence of low-frequency categories and hence the main source of privacy risk. In the regime $\sigma\in(0,1)$, the mechanism satisfies instance-level $(\varepsilon,\delta)$-DP. Stronger guarantees are obtained in the two special regimes: for $\sigma<0$, corresponding to a parametric Dirichlet--Multinomial model, the mechanism satisfies instance-level $\varepsilon$-DP and, under suitable conditions on the released sample size, global $\varepsilon$-DP \citep{machana}; for $\sigma=0$, corresponding to the Dirichlet process prior, it satisfies global $(\varepsilon,\delta)$-DP, with $\delta$ controlled by the ratio between the synthetic and confidential sample sizes. These results show that DP guarantees are strongly tied to the predictive structure induced by the prior, suggesting the study of alternative priors and privacy notions better adapted to posterior-predictive mechanisms.

The utility results provide a frequentist validation of the PY posterior-predictive mechanism. Following the notion of informativity in \cite{wasserman}, we measure utility by the expected $1$-Wasserstein distance between the synthetic empirical distribution and the ``true'' data-generating distribution. In the regimes $\sigma=0$ and $\sigma<0$, we prove consistency in this metric and derive explicit convergence rates. These rates quantify the privacy--utility tradeoff: stronger privacy requirements may restrict the released sample size and slow convergence to the ``true'' distribution. When $(\varepsilon,\delta)$-DP is imposed only asymptotically, however, synthetic data achieve the same consistency rate as the confidential data. As a byproduct, we obtain a concentration bound for the posterior Dirichlet process around its posterior mean in expected $1$-Wasserstein distance, leaving analogous results for other Bayesian nonparametric priors as an open question.

%%%%%%%%%%%%%%%%%%%%%%%%%%%%%%%%
%%%%%%%%%%%%%%%%%%%%%%%%%%%%%%%%
%%%%%%%%%%%%%%%%%%%%%%%%%%%%%%%%
%%%%%%%%%%%%%%%%%%%%%%%%%%%%%%%%

\bibliographystyle{apalike}
\bibliography{references}
%%%%%%%%%%%%%%%%%%%%%%%%%%%%%%%%
%%%%%%%%%%%%%%%%%%%%%%%%%%%%%%%%
%%%%%%%%%%%%%%%%%%%%%%%%%%%%%%%%
%%%%%%%%%%%%%%%%%%%%%%%%%%%%%%%%
\clearpage

\section*{Supplementary material to: Bayesian Nonparametric Privacy-Preserving Synthetic Data Generation}\label{SM}

\appendix
 \setcounter{page}{1}

\numberwithin{equation}{section}
\numberwithin{theorem}{section}
\numberwithin{lemma}{section}

%%%%%%%%%%%%%%%%%%%%%%%%%%%%%%%%
%%%%%%%%%%%%%%%%%%%%%%%%%%%%%%%%
%%%%%%%%%%%%%%%%%%%%%%%%%%%%%%%%
%%%%%%%%%%%%%%%%%%%%%%%%%%%%%%%%
\section{Additional theory}

\subsection{Additional background material on Pitman--Yor process}
\begin{comment}
The Pitman--Yor process can also be defined a species sampling model characterized by the non-atomic distribution $H$ and the following sequence of exchangeable partition probability functions (EPPF):
\begin{equation*}
\label{eppfPY}
    p_k^{(n)}(n_1,...,n_k) = \dfrac{\prod_{i=1}^{k-1}(\theta + i \sigma)}{(\theta + 1)_{n-1}} \prod_{i=1}^{k} ( 1 - \sigma)_{n_i-1},
\end{equation*}
where $(a)_b := \Gamma(a+b)/\Gamma(a)$ for $a,b \geq 0$ denotes the Pochhammer symbol and the parameters $\theta,\sigma$ satisfy either $\sigma \in [0,1)$ and $\theta > - \sigma $, or $\sigma < 0$ and $\theta = z |\sigma|$ for some integer $z \geq 1$. We need this characterization in the proof of Theorem \ref{theorem_privacy_main}. 
\end{comment}
The update rule of the Pitman--Yor process is given by the following proposition (\cite{pitmanSampling}):
\begin{proposition}
\label{pitman posterior}
    Let $(X_1,...,X_n)$ be a sample from a random probability measure $\tilde{p} \sim \mathscr{PY}(\sigma,\theta,H)$. Then, conditionally on $\mathbf{X}$ featuring $k$ distinct values $X_i^*$ with frequencies $n_i$ for $i=1,...,k$, 
    \begin{equation*}
        \tilde{p} \mid \mathbf{X} \overset{\text{d}}{=} \sum_{j=1}^k \tilde{P}_j \delta_{X_j^*} + \tilde{R}_k F_k,
    \end{equation*}
    where $(\tilde{P}_1,...,\tilde{P}_k, \tilde{R}_k)$ has $\text{Dirichlet}(n_1 - \sigma,...,n_k - \sigma, \theta + k\sigma)$ distribution, independently of the random probability measure $F_k$, which has $\mathscr{PY}(\sigma, \theta + k \sigma, H)$-distribution. 
\end{proposition}

\begin{comment}
\subsection{A factorial-moment version of Markov's inequality}

For \(q \in \mathbb N\), let
\[
x^{\underline q}
=
x(x-1)\cdots(x-q+1)
\]
denote the falling factorial, with the convention \(x^{\underline 0}=1\).

\begin{lemma}[\cite{naveau1997comparison}]\label{l1}
Let \(X\) be a nonnegative integer-valued random variable. For any integers \(k \ge q \ge 1\),
\[
\mathbb P(X \ge k)
\le
\frac{\mathbb E\left[X^{\underline q}\right]}{k^{\underline q}}.
\]
\end{lemma}

To apply this inequality to Beta--binomials we use the next identity, which is standard and follows immediately from the Beta mixture representation of the Beta--Binomial distribution.

\begin{lemma}[Factorial moments of the Beta--Binomial law]\label{l2}
If \(X \sim \mathrm{BetaBin}(m,a,b)\), then for every integer \(q \in \{1,\dots,m\}\),
\[
\mathbb E\left[X^{\underline q}\right]
=
m^{\underline q}\frac{(a)_q}{(a+b)_q}
\]
\end{lemma}
\end{comment}

\subsection{Other privacy mechanisms}
\label{histograms}
We describe three privacy mechanisms known in the literature. They all rely on generating synthetic data by sampling from a histogram.

We consider $\mathbb{X} = [0,1]^r$ for some integer $r \geq 1$ and partition $\mathbb{X}$ into $k$ bins ${B_1,...,B_k}$, where each bin $B_j$ is a cube with sides of length $h \in (0,1)$ such that $k = 1/h^r$ is an integer. We denote by $\hat{f}_k(x)$ the histogram estimator for the confidential data $\mathbf{X}$, defined as
\begin{equation*}
    \hat{f}_k(x) := \sum_{j=1}^k \dfrac{\hat{p}_j}{h^r} \mathbbm{1}_{B_j}(x),
\end{equation*}
where $\hat{p}_j = C_j/n$ and $C_j = \sum_{i=1}^n \mathbbm{1}_{B_j}(X_i)$. 

\textbf{(1) Smoothed estimator mechanism, \cite{wasserman}} The histogram estimator can be smoothed in the following way:
\begin{equation*}
    \hat{f}_{k,\sigma}(x) = (1 - \sigma) \hat{f}_{k}(x) + \sigma,
\end{equation*}
where $\sigma \in (0,1)$. Sampling from the smoothed histogram satisfies $\varepsilon$-differential privacy if 
\begin{equation}\label{smoothedDP}
    m \log \left (  \dfrac{(1 - \sigma)k}{n\sigma}  + 1 \right ) \leq \varepsilon.
\end{equation}

\textbf{(2) Perturbed histogram estimator, \cite{wasserman}} One can instead perturb the set of counts by adding noise, build a new histogram estimator based on the perturbed counts and then release data sampling from it. Formally, let $D_j = C_j + \nu_j$, where $\nu_1,...,\nu_m \iid \text{Laplace}(0,2/\varepsilon)$. Let $\tilde{D}_j = \max\{D_j,0\}$, $\hat{q}_j = \tilde{D}_j/\sum_{s=1}^k \tilde{D}_s$  for $j=1,...,k$ and define
\begin{equation*}
    \tilde{f}(x) := \sum_{j=1}^k \dfrac{\hat{q}_j}{h^r} \mathbbm{1}_{B_j}(x).
\end{equation*}
Any sample from $\tilde{f}(x)$ preserves $\varepsilon$-differential privacy for any $k$.

\textbf{(3) Dirichlet-Multinomial mechanism, \cite{machana}} Let $\mathbf{p} := (p_1,...,p_k)$ be the vector  of probabilities $p_j := \mathbb{P}[X_i \in B_j]$ for $j =1,...,k$, and $(C_1,...,C_k)$ the bin counts. The mechanism relies on the following Bayesian parametric model:
\begin{align*}
    \mathbf{p}\sim\text{Dirichlet}(\alpha_1,...,\alpha_k)\\
    C_1,...,C_k|\mathbf{p}\sim\text{Multinomial$(n,\mathbf{p})$}.
\end{align*}
The mechanism involves drawing a vector of probabilities $\mathbf{q}:= (q_1,...,q_k)$ from the posterior of the model, which is still a Dirichlet with updated parameters $(\alpha_1 + C_1,..., \alpha_k + C_k)$, and generating a synthetic vector of bin counts $\mathbf{D} := (D_1,...,D_k)$ sampling from a Multinomial$(m,\mathbf{q})$.
Such mechanism satisfies $\varepsilon$-differential privacy if 
\begin{equation}\label{dm_cond}
    \alpha_j \geq m / (e^\varepsilon - 1)\text{ for all } j=1,...,k.  
\end{equation}

\subsection{Rate of convergence in Wasserstein distance of the
empirical measure}
\label{empirical}
We assume $X_1,...,X_n \iid \mathfrak{p}_0$. $p$-Wasserstein convergence of the empirical measure $e_n^{(\mathbf{X})}$ to $\mathfrak{p_0}$ can be quantified by providing a non-asymptotic estimate of the type
\begin{equation}
\label{rate_empirical}
    \varepsilon_{n,p}(\mathbb{X},\mathfrak{p}_0) := \mathbb{E}_{\otimes^n  \mathfrak{p}_0}[ \mathcal{W}_p^{\mathcal{P}(\mathbb{X})}(e_n^{(\mathbf{X})}, \mathfrak{p}_0)] \leq g(n)
\end{equation}
for all $n \geq 1$, where $g$ is a suitable function. The following theorem provides such an estimate in the case $\mathbb{X} \subseteq \mathbb{R}^d $. 
\begin{theorem}[\cite{fournier2015rate}]
\label{fournier}
    Let $\mathfrak{p}_0$ a probability measure on $\mathbb{R}^d$ and let $p \geq 1$. Assume that
    \begin{equation*}
        M_q(\mathfrak{p}_0) := \int_{\mathbb{X}} | x |^q  \ \mathfrak{p}_0(dx) < \infty
    \end{equation*}
    for some $q > p$. Then there exists a constant $C$ depending only on $p,d,q$ such that for all $n \geq 1$,
\begin{equation*}
\varepsilon_{n,p}(\mathbb{X},\mathfrak{p}_0) 
\leq C \, M_q^{p/q}(\mathfrak{p}_0) \times\ 
\begin{cases}     
n^{-1/2} + n^{-(q-p)/q}, & \text{if } p > d/2 \text{ and } q \neq 2p \\[1ex]
n^{-1/2} \log(1+n) + n^{-(q-p)/q}, & \text{if } p = d/2 \text{ and } q \neq 2p \\[1ex]
n^{-p/d} + n^{-(q-p)/q}, & \text{if } p \in [1, d/2) \text{ and } q \neq d/(d-p)
\end{cases}
\end{equation*}

\end{theorem}
If $\mathfrak{p}_0$ ha sufficiently many moments, then the term $n^{-(q-p)/q}$ is small and can be neglected. 
We need this result when proving Theorems \ref{consistenza multinomiale} and \ref{consistencydir}, where we have to provide rates for the expected 1-Wasserstein distance between the empirical measure of the synthetic data and the data generating distribution. 

\subsection{Bayesian consistency}
To prove consistency in the Dirichlet-Multinomial subcase we need the following result, arising from the theory of Bayesian consistency and posterior contraction rates (\cite{GhosalvanderVaart2017}).
Being in the case of Bayesian parametric models, consider the metric space $(\Theta, d_{\Theta})$, where $\Theta$ is the parameter space. We denote $\mathcal{W}_p^{\mathcal{P}(\Theta)}(\cdot , \cdot )$ the $p$-Wasserstein distance defined in $(\Theta, d_{\theta})$, assuming it is a complete and separable metric space. 

Let $X_1,..., X_n$ be a sequence of categorical random variables taking values in the finite set $\{a_1,...,a_k\} \subset \mathbb{X}$. The parameter space $\Theta$, in this case, coincides with the interior of the $k-1$ dimensional simplex 
\begin{equation*}
    \mathscr{S}^{k-1} = \{ (\theta_1,...,\theta_{k-1}) \in [0,1]^{k-1} : \sum_{i=1}^{k-1} \theta_i \leq 1 \}.
\end{equation*}
\begin{proposition}[\cite{dolera2023strong}]
\label{pcrPara}
    Let $k \geq 2$. Let $\phi$ a prior on $\mathscr{S}^{k-1}$. If $\phi$ has a Lebesgue density $q$ such that $q \in C^1(\overline{\mathscr{S}^{k-1}})$ and $q(\theta) = 0$ for any $\theta \in \partial \mathscr{S}^{k-1}$, then as $n \rightarrow \infty$
    \begin{equation*}
       \mathbb{E}_{\otimes^n \mathfrak{p}_0} \left[ \mathcal{W}_2^{\mathcal{P}(\Theta)}(\psi_n( \ \cdot \mid \mathbf{X}), \delta_{\theta_0})  \right] = O(n^{-1/2}),
    \end{equation*}
    which is the optimal rate.
\end{proposition}
%% qualcosina sul risultato sui modelli dominati 
Notice that the above proposition holds assuming a Dirichlet prior with parameters $(\alpha_1,...,\alpha_k)$ such that $\alpha_i \geq 2$ for all $i=1,...,k$.
\subsection{Lipschitz continuity of the Dirichlet-Multinomial operator}
We state and prove a result, needed in the proof of Theorem \ref{consistenza multinomiale}, concerning a property of the operator $\mathscr{T} : (\Theta, d_{\Theta}) \rightarrow  (\mathcal{P} (\mathbb{X}),\mathcal{W}_1^{\mathcal{P}(\mathbb{X})})$, which for a Dirichlet-Multinomial model is defined as
\begin{equation*}
    \mathscr{T}(\theta) = \sum_{i=1}^k \theta_i \delta_{a_i}
\end{equation*}
for all $\theta \in \Theta \equiv \mathscr{S}^{k-1}$, where $\theta_k := 1 - \theta_1 - \cdots - \theta_{k-1}$. We assume $d_{\Theta}$ equal to the metric induced by the $l^2$ norm in $\mathbb{R}^{k-1}$.
\begin{lemma}
\label{lemmino}
    The map $\mathscr{T} : (\Theta, d_{\Theta}) \rightarrow  (\mathcal{P} (\mathbb{X}),\mathcal{W}_1^{\mathcal{P}(\mathbb{X})})$ defined above is Lipschitz continuous, with Lipschitz constant less or equal to $\sqrt{k-1}D$, with $D:= \text{diam}(\mathbb{X})$. 
\end{lemma}
\begin{proof}
For any $\theta,\phi \in \Theta$, by \cite{villani2008optimal} (Theorem 6.15), we have
\begin{equation*}
    \mathcal{W}_1^{\mathcal{P}(\mathbb{X})}(\mathscr{T}(\theta),\mathscr{T}(\phi)) \leq D \ \lVert \mathscr{T}(\theta) - \mathscr{T}(\phi) \rVert_{\text{TV}}. 
\end{equation*}
By \cite{roch_mdp_2024} (Lemma 4.1.9), we can say
\begin{equation*}
    \lVert \mathscr{T}(\theta) - \mathscr{T}(\phi) \rVert_{\text{TV}} = \dfrac{1}{2} \sum_{i=1}^k | \theta_i - \phi_i | ,
\end{equation*}
where
\begin{equation*}
\begin{aligned}
    \sum_{i=1}^k | \theta_i - \phi_i | = & \sum_{i=1}^{k-1} | \theta_i - \phi_i |  + | \theta_k - \phi_k | =  \\ &  \sum_{i=1}^{k-1} | \theta_i - \phi_i | + | - \theta_{k-1} + \phi_{k-1} + ...  - \theta_{1} + \phi_{1} | \leq \\ & 2 \sum_{i=1}^{k-1} | \theta_i - \phi_i |  \leq 2 \sqrt{k-1} \ \lVert \theta - \phi \lVert_{l^2}
\end{aligned}
\end{equation*}
by triangular inequality and Cauchy-Schwarz inequality.
\end{proof}

\section{Proofs}
\subsection{Proof of Theorem \ref{theorem_privacy_main}}
By Proposition 2, it is enough to prove instance-level \((\varepsilon,\delta)\)-probabilistic differential privacy. Work under the fixed-dataset notation introduced above. For ease of notation let $K = K_m^{(n)}$.
Fix a neighbouring dataset \(\tilde{\mathbf x}\) such that \(h(\mathbf x,\tilde{\mathbf x})=1\). By exchangeability, assume that the datasets differ only in the first coordinate:
\[
x_1=x_l^*, \qquad \tilde x_1=b\neq x_l^*, \qquad x_i=\tilde x_i,\quad i=2,\ldots,n.
\]
We write \(q_{\mathbf x}\) and \(q_{\tilde{\mathbf x}}\) for densities of \(Q_{\mathbf x}\) and \(Q_{\tilde{\mathbf x}}\) with respect to a common dominating measure. Define
\[
A_{\mathbf x,\tilde{\mathbf x}} := \left\{ z\in \mathbb X^m: q_{\mathbf x}(z)\le e^\varepsilon q_{\tilde{\mathbf x}}(z) \right\}.
\]
We now control
\[
Q_{\mathbf x}(A_{\mathbf x,\tilde{\mathbf x}}^c)
\]
for the two possible types of value \(b\).

\paragraph{Case 1: \(b\) is already present in \(\mathbf x\).}

Assume
\[
b=x_t^* \qquad\text{for some }t\neq l.
\]

First suppose that \(n_l\ge 2\). Then \(x_l^*\) remains an old species under \(\tilde{\mathbf x}\), with frequency \(n_l-1\), while \(x_t^*\) has frequency \(n_t+1\). The number of old species is the same under \(\mathbf x\) and \(\tilde{\mathbf x}\). Using the posterior-predictive density induced by the Pitman--Yor EPPF, we obtain
\[
\frac{q_{\mathbf x}(z)}{q_{\tilde{\mathbf x}}(z)} = \frac{(n_t-\sigma)_{S_t}}{(n_t+1-\sigma)_{S_t}} \frac{(n_l-\sigma)_{S_l}}{(n_l-1-\sigma)_{S_l}}.
\]
Equivalently,
\[
\frac{q_{\mathbf x}(z)}{q_{\tilde{\mathbf x}}(z)} = \frac{n_t-\sigma}{n_t+S_t-\sigma} \frac{n_l+S_l-1-\sigma}{n_l-1-\sigma}.
\]
Therefore, in this subcase,
\[
A_{\mathbf x,\tilde{\mathbf x}}^c \subseteq \left\{ \frac{n_t-\sigma}{n_t+S_t-\sigma} \frac{n_l+S_l-1-\sigma}{n_l-1-\sigma} > e^\varepsilon \right\}.
\]

Now suppose instead that \(n_l=1\). Then the value \(x_l^*\) disappears from \(\tilde{\mathbf x}\). Hence, under \(Q_{\tilde{\mathbf x}}\), the fixed value \(x_l^*\) is not an old atom. Since \(H\) is nonatomic,
\[
Q_{\tilde{\mathbf x}} \left( \exists r\le m: Z_r=x_l^* \right) = 0.
\]
Thus the event
\[
\{S_l\ge 1\}
\]
is singular for \(Q_{\mathbf x}\) with respect to \(Q_{\tilde{\mathbf x}}\), and it must be included in the bad privacy event.

On the complementary event \(\{S_l=0\}\), the value \(x_l^*\) is not released. Then the released sample contains the same new species under both \(\mathbf x\) and \(\tilde{\mathbf x}\), but \(\tilde{\mathbf x}\) has one fewer old species. The likelihood ratio becomes
\[
\frac{q_{\mathbf x}(z)}{q_{\tilde{\mathbf x}}(z)} = \frac{\theta+(j+K-1)\sigma}{\theta+(j-1)\sigma} \frac{n_t-\sigma}{n_t+S_t-\sigma}.
\]
Consequently, when \(n_l=1\),
\[
A_{\mathbf x,\tilde{\mathbf x}}^c \subseteq \{S_l\ge 1\} \cup \left( \{S_l=0\} \cap \left\{ \frac{\theta+(j+K-1)\sigma}{\theta+(j-1)\sigma} \frac{n_t-\sigma}{n_t+S_t-\sigma} > e^\varepsilon \right\} \right).
\]

\paragraph{Case 2: \(b\) is not present in \(\mathbf x\).}

Assume
\[
b\notin\{x_1^*,\ldots,x_j^*\}.
\]
Let
\[
E_b:=\{\exists r\le m: Z_r=b\}.
\]
Since \(b\) is fixed before generating \(Z\), and since \(H\) is nonatomic,
\[
Q_{\mathbf x}(E_b)=0.
\]
Thus we may work on \(E_b^c\), up to a \(Q_{\mathbf x}\)-null event.

First suppose that \(n_l\ge 2\). Then \(x_l^*\) remains an old species under \(\tilde{\mathbf x}\), with frequency \(n_l-1\), and \(b\) is an old singleton under \(\tilde{\mathbf x}\) with released count zero on \(E_b^c\). The likelihood ratio is
\[
\frac{q_{\mathbf x}(z)}{q_{\tilde{\mathbf x}}(z)} = \frac{\theta+j\sigma}{\theta+(j+K)\sigma} \frac{(n_l-\sigma)_{S_l}}{(n_l-1-\sigma)_{S_l}}.
\]
Equivalently,
\[
\frac{q_{\mathbf x}(z)}{q_{\tilde{\mathbf x}}(z)} = \frac{\theta+j\sigma}{\theta+(j+K)\sigma} \frac{n_l+S_l-1-\sigma}{n_l-1-\sigma}.
\]
Hence, since \(Q_{\mathbf x}(E_b)=0\),
\[
Q_{\mathbf x}(A_{\mathbf x,\tilde{\mathbf x}}^c) \le Q_{\mathbf x} \left( \frac{\theta+j\sigma}{\theta+(j+K)\sigma} \frac{n_l+S_l-1-\sigma}{n_l-1-\sigma} > e^\varepsilon \right).
\]

Now suppose that \(n_l=1\). Then \(\tilde{\mathbf x}\) is obtained by replacing the singleton old species \(x_l^*\) by the singleton old species \(b\). On the event \(\{S_l\ge 1\}\), the released sample contains \(x_l^*\), which is not an old atom under \(\tilde{\mathbf x}\). Since \(H\) is nonatomic,
\[
Q_{\tilde{\mathbf x}}(S_l\ge 1)=0.
\]
Thus \(\{S_l\ge 1\}\) is again singular.

On the event
\[
\{S_l=0\}\cap E_b^c,
\]
neither \(x_l^*\) nor \(b\) appears in the released sample. The two datasets then differ only by the label of an old singleton with released count zero. The old-species contribution is
\[
(1-\sigma)_0=(1-\sigma)_0=1,
\]
the number of old species is \(j\) for both datasets, and the new synthetic species are the same under both descriptions. Therefore
\[
\frac{q_{\mathbf x}(z)}{q_{\tilde{\mathbf x}}(z)} = 1 \qquad \text{on } \{S_l=0\}\cap E_b^c.
\]
Since \(Q_{\mathbf x}(E_b)=0\), we conclude that, in this subcase,
\[
Q_{\mathbf x}(A_{\mathbf x,\tilde{\mathbf x}}^c) \le Q_{\mathbf x}(S_l\ge 1).
\]

Then, the exact pairwise bad events are:
\[
\mathcal E_{l,t}^{\mathrm{old}}(\varepsilon) :=
\begin{cases}
\left\{ \dfrac{n_t-\sigma}{n_t+S_t-\sigma} \dfrac{n_l+S_l-1-\sigma}{n_l-1-\sigma} > e^\varepsilon \right\}, & n_l\ge 2, \\
\{S_l\ge 1\} \cup \left( \{S_l=0\} \cap \left\{ \dfrac{\theta+(j+K-1)\sigma}{\theta+(j-1)\sigma} \dfrac{n_t-\sigma}{n_t+S_t-\sigma} > e^\varepsilon \right\} \right), & n_l=1,
\end{cases}
\]
for neighbors changing $x^*_l$ into an already observed $x^*_t$, and
\[
\mathcal E_l^{\mathrm{new}}(\varepsilon) :=
\begin{cases}
\left\{ \dfrac{\theta+j\sigma}{\theta+(j+K)\sigma} \dfrac{n_l+S_l-1-\sigma}{n_l-1-\sigma} > e^\varepsilon \right\}, & n_l\ge 2, \\
\{S_l\ge 1\}, & n_l=1,
\end{cases}
\]
for neighbors changing $x^*_l$ into a value not already present in $\mathbf x$.

Thus,
\[
\tilde \delta(\varepsilon,\mathbf x) := \max\left\{ \max_{1\le l\le j} Q_{\mathbf x}\left(\mathcal E_l^{\mathrm{new}}(\varepsilon)\right),  \max_{\substack{1\le l,t\le j\\ t\neq l}} Q_{\mathbf x}\left(\mathcal E_{l,t}^{\mathrm{old}}(\varepsilon)\right) \right\},
\]
\color{black}

\subsection{Proof of Corollary \ref{cor_instance_dp_dirichlet}}
Set \(\sigma=0\) in Theorem 1. If \(n_l\ge2\), the bad event for changing \(x_l^*\) into a new value is
\[
\left\{ \frac{n_l+S_l-1}{n_l-1}>e^\varepsilon \right\} = \left\{ S_l>(e^\varepsilon-1)(n_l-1) \right\} = \left\{ S_l\ge k_l(\varepsilon) \right\}.
\]
The bad event for changing \(x_l^*\) into an old value is contained in the same event, since \(n_t/(n_t+S_t)\le1\). If \(n_l=1\), both old-value and new-value changes reduce to the singular event \(\{S_l\ge1\}=\{S_l\ge k_l(\varepsilon)\}\), because the remaining likelihood-ratio condition in the old-value case is bounded by \(1<e^\varepsilon\). Hence Theorem 1 applies with
\[
\max_{1\le i\le j} Q_{\mathbf x}\left(S_i\ge k_i(\varepsilon)\right)
\]
in place of \(\delta^\sharp(\varepsilon,\mathbf x)\).
Under the Dirichlet posterior-predictive distribution, conditionally on \(K_n=j\) and \(\mathbf N_n=\mathbf n\),
\[
S_i \sim \mathrm{BetaBin}\left(m,n_i,\theta+n-n_i\right).
\]
Therefore
\[
Q_{\mathbf x}\left(S_i\ge k_i(\varepsilon)\right) = \sum_{s=k_i(\varepsilon)}^m \binom ms \frac{ (n_i)_s (\theta+n-n_i)_{m-s} }{ (\theta+n)_m },
\]
which gives the stated \(\delta^\sharp_0(\varepsilon,\mathbf x)\).

\subsection{Proof of Theorem \ref{cor_dp}}
By the proof of Corollary \ref{cor_instance_dp_dirichlet}, for the fixed dataset \(\mathbf x\), the mechanism satisfies instance-level \((\varepsilon,\delta)\)-differential privacy for every
\[
\delta> \delta^\sharp_0(\varepsilon,\mathbf x) = \max_{1\le i\le j} Q_{\mathbf x} \left( S_i\ge k_i(\varepsilon) \right),
\]
Moreover, conditionally on \(K_n=j\) and \(\mathbf N_n=\mathbf n\),
\[
S_i \sim \mathrm{BetaBin} \left( m, n_i, \theta+n-n_i \right).
\]

Therefore the instance-level privacy parameter depends on \(\mathbf x\) only through the multiplicities \(n_i\). For \(r=1,\ldots,n\), let
\[
Y_{r,n,m}\sim \mathrm{BetaBin} \left( m, r, \theta+n-r \right),
\]
and set
\[
\delta^{\mathrm G}_{n,m}(\varepsilon) := \max_{1\le r\le n} \mathbb P \left( Y_{r,n,m}\ge k(r, \varepsilon) \right),
\]
where $k(r, \varepsilon)$ is as in \eqref{eq:ki_def}.
Then, for every dataset \(\mathbf x\), $\delta^\sharp_0(\varepsilon,\mathbf x) \le \delta^{\mathrm G}_{n,m}(\varepsilon)$.
Since \(\mathbf x\) was arbitrary, the instance-level guarantee holds uniformly over all datasets. Hence, for every pair \(\mathbf x,\tilde{\mathbf x}\in\mathcal X^n\) with \(h(\mathbf x,\tilde{\mathbf x})=1\), and every measurable \(B\in\mathcal X^m\),
\[
Q_{\mathbf x}(B) \le e^\varepsilon Q_{\tilde{\mathbf x}}(B) + \delta
\]
for every
\[
\delta>\delta^{\mathrm G}_{n,m}(\varepsilon).
\]
Thus the mechanism satisfies global \((\varepsilon,\delta)\)-differential privacy for every
\[
\delta\in[0,1) \quad\text{such that}\quad \delta>\delta^{\mathrm G}_{n,m}(\varepsilon).
\]

It remains to express and upper bound \(\delta^{\mathrm G}_{n,m}(\varepsilon)\). Since the beta-binomial probability mass function is
\[
\mathbb P(Y_{r,n,m}=s) = \binom ms \frac{ (r)_s(\theta+n-r)_{m-s} }{ (\theta+n)_m }, \qquad s=0,\ldots,m,
\]
we have
\[
\delta^{\mathrm G}_{n,m}(\varepsilon) = \max_{1\le r\le n} \sum_{s=k_r(\varepsilon)}^m \binom ms \frac{ (r)_s(\theta+n-r)_{m-s} }{ (\theta+n)_m },
\]
with the convention that the sum is zero whenever \(k_r(\varepsilon)>m\).

We now prove the simpler global upper bound. First consider \(r=1\). Then
\[
\begin{aligned}
\mathbb P(Y_{1,n,m}\ge 1) &= 1-\mathbb P(Y_{1,n,m}=0)                                     \\
&= 1- \frac{(\theta+n-1)_m}{(\theta+n)_m}                            \\
&= 1- \frac{\theta+n-1}{\theta+n+m-1}                                \\
&= \frac{m}{\theta+n+m-1}.
\end{aligned}
\]
This is the singular singleton contribution.

Now consider \(r\ge2\). Put
\[
a_\varepsilon:=e^\varepsilon-1.
\]
By definition,
\[
k_r(\varepsilon) = \left\lfloor a_\varepsilon(r-1) \right\rfloor+1.
\]
Since \(Y_{r,n,m}\) is integer-valued,
\[
\{Y_{r,n,m}\ge k_r(\varepsilon)\} = \{Y_{r,n,m}>a_\varepsilon(r-1)\}.
\]
Moreover,
\[
\mathbb E[Y_{r,n,m}] = \frac{mr}{\theta+n}.
\]
Therefore, by Markov's inequality,
\[
\begin{aligned}
\mathbb P \left( Y_{r,n,m}\ge k_r(\varepsilon) \right) &= \mathbb P \left( Y_{r,n,m}>a_\varepsilon(r-1) \right)                                                       \\
&\le \frac{ \mathbb E[Y_{r,n,m}] }{ a_\varepsilon(r-1) }                                                             \\
&= \frac{mr} {(\theta+n)(e^\varepsilon-1)(r-1)}                             \\
&\le \frac{2m} {(\theta+n)(e^\varepsilon-1)},
\end{aligned}
\]
where the last inequality uses \(r/(r-1)\le2\) for \(r\ge2\).

Combining the cases \(r=1\) and \(r\ge2\), we obtain
\[
\delta^{\mathrm G}_{n,m}(\varepsilon) \le \max\left\{ \frac{m}{\theta+n+m-1}, \frac{2m}{(\theta+n)(e^\varepsilon-1)} \right\}.
\]
Thus \(\delta^{\mathrm G}_{n,m}(\varepsilon)\le\bar\delta_{n,m}(\varepsilon)\), and every \(\delta>\bar\delta_{n,m}(\varepsilon)\) is admissible.
Hence, for fixed \(\theta\) and \(\varepsilon\),
\[
\bar\delta_{n,m}(\varepsilon) = O\left(\frac mn\right).
\]
This completes the proof.

\subsection{Proof of Corollary \ref{easyprop}}
  In this case, the mechanism cannot generate new unique values, because $z = j$ and $\theta = z |\sigma |$, with $\sigma < 0$. Therefore the proof goes as that of Theorem \ref{theorem_privacy_main}, but only Case $1$ can occur. Moreover, no singularity for $Q_{\mathbf{x}}$ with respect to $Q_{\mathbf{\tilde{x}}}$ happens. It follows that to satisfy instance-level $(\varepsilon,\delta)$-differential privacy it is enough to require
  \begin{equation}
  \label{ddd}
      \dfrac{n_t + |\sigma|}{n_t + S_t + |\sigma| } \dfrac{n_l + S_l - 1 + |\sigma|}{n_l - 1 + |\sigma|} \leq e^{\varepsilon}
  \end{equation}
  for all $t,l=1,...,j$ such that $t \neq l$. Since $\sigma$ can take any negative value, we require $n_i + |\sigma| - 1 \geq m/(e^\varepsilon-1)$ for all $i=1,...,j$ and (\ref{ddd}) is satisfied with probability $1$, i.e. $\delta = 0$. Thus if
  \begin{equation}
    m \leq (| \sigma | + n_i - 1)(e^\varepsilon - 1)
\end{equation}
holds for all $i = 1,...,j$, then we have instance-level $\varepsilon$-differential privacy.

\subsection{Proof of Theorem \ref{consistenza multinomiale}}
    The mechanism $Q_n$ has the following density with respect to the counting measure:
    \begin{equation*}
        q_n(\mathbf{z} \mid \mathbf{X} = \mathbf{x}) = \int_{\Theta} \left [ \prod_{j=1}^m f_\theta(z_j) \right ] \psi_n( d \theta \mid \mathbf{X} = \mathbf{x}),
    \end{equation*}
    where $f_\theta(z) = \left( \prod_{i=1}^k \theta_i^{\mathbbm{1}_{\{a_i\}}(z)} \right)  \mathbbm{1}_{\{a_1,...,a_k\}}(z)$ is the density of a Multinomial with support $\{a_1,...,a_k\}$, and $\psi_n$ is the posterior. Hence, by Fubini-Tonelli's theorem, we have 
    \begin{equation*}
    \begin{aligned}
         & \mathbb{E}_{\nu_{n,m}}\left [\mathcal{W}_1^{\mathcal{P}(\mathbb{X})}(e_m^{(\mathbf{Z})},\mathfrak{p}_0) \right ] = 
         \int_{\mathbb{X}^m \times \mathbb{X}^n} \mathcal{W}_1^{\mathcal{P}(\mathbb{X})}(e_m^{(\mathbf{z})},\mathfrak{p}_0) \ q_n(\mathbf{z} \mid \mathbf{X} = \mathbf{x}) \ d\mathbf{z} \ \mathfrak{p}_0^n(d\mathbf{x}) = \\ &
          \int_{\mathbb{X}^m \times \mathbb{X}^n} \left ( \int_{\Theta} \mathcal{W}_1^{\mathcal{P}(\mathbb{X})}(e_m^{(\mathbf{z})},\mathfrak{p}_0) \left [ \prod_{j=1}^m f_\theta(z_j) \right ] \psi_n( d \theta \mid \mathbf{X} = \mathbf{x})   \right ) d\mathbf{z} \ \mathfrak{p}_0^n(d\mathbf{x})  = \\ &
        \int_{\Theta \times \mathbb{X}^n } \left ( \int_{\mathbb{X}^m} \mathcal{W}_1^{\mathcal{P}(\mathbb{X})}(e_m^{(\mathbf{z})},\mathfrak{p}_0) \left [ \prod_{j=1}^m f_\theta(z_j) \right ]  d\mathbf{z}    \right ) \psi_n( d \theta \mid \mathbf{X} = \mathbf{x})  \ \mathfrak{p}_0^n(d\mathbf{x}) = \\
        &\mathbb{E}_{\psi_n \otimes \mathfrak{p}_0^n} \left [\mathbb{E} \left [ \mathcal{W}_1^{\mathcal{P}(\mathbb{X})}(e_m^{(\mathbf{Z})},\mathfrak{p}_0) \mid \theta \right ] \right ].
    \end{aligned}
    \end{equation*}
    By triangular inequality, we have
    \begin{equation*}
        \mathbb{E}_{\psi_n \otimes \mathfrak{p}_0^n} \left [\mathbb{E} \left [ \mathcal{W}_1^{\mathcal{P}(\mathbb{X})}(e_m^{(\mathbf{Z})},\mathfrak{p}_0) \mid \theta \right ] \right ] \leq  \mathbb{E}_{\psi_n \otimes \mathfrak{p}_0^n} \left [\mathbb{E} \left [ \mathcal{W}_1^{\mathcal{P}(\mathbb{X})}(e_m^{(\mathbf{Z})},\mathscr{T}(\theta)) \mid \theta \right ] \right ] +
         \mathbb{E}_{\psi_n \otimes \mathfrak{p}_0^n} \left [\mathbb{E} \left [ \mathcal{W}_1^{\mathcal{P}(\mathbb{X})}(\mathscr{T}(\theta),\mathfrak{p}_0) \mid \theta \right ] \right ].
    \end{equation*}
    By the property of the conditional expectation
    \begin{equation*}
        \mathbb{E}_{\psi_n \otimes \mathfrak{p}_0^n} \left [\mathbb{E} \left [ \mathcal{W}_1^{\mathcal{P}(\mathbb{X})}(\mathscr{T}(\theta),\mathfrak{p}_0) \mid \theta \right ] \right ] =  \mathbb{E}_{\psi_n \otimes \mathfrak{p}_0^n} \left [\mathcal{W}_1^{\mathcal{P}(\mathbb{X})}(\mathscr{T}(\theta),\mathfrak{p}_0) \right ].
    \end{equation*}
    We can write
    \begin{equation*}
        \mathbb{E}_{\psi_n \otimes \mathfrak{p}_0^n} \left [\mathcal{W}_1^{\mathcal{P}(\mathbb{X})}(\mathscr{T}(\theta),\mathfrak{p}_0) \right ] = \mathbb{E}_{\mathfrak{p}_0^n} \left [\int_{\Theta^2} \mathcal{W}_1^{\mathcal{P}(\mathbb{X})}(\mathscr{T}(\theta),\mathscr{T}(\phi)) \ \psi_n(d \theta | \mathbf{X} = \mathbf{x}) \ \delta_{\theta_0}(d \phi) \right ].
    \end{equation*}
    By Lemma \ref{lemmino} and H\"older's inequality, we obtain 
    \begin{equation*}
    \begin{aligned}
        & \mathbb{E}_{ \otimes^n \mathfrak{p}_0} \left [\int_{\Theta^2} \mathcal{W}_1^{\mathcal{P}(\mathbb{X})}(\mathscr{T}(\theta),\mathscr{T}(\phi)) \ \psi_n(d \theta \mid \mathbf{X} = \mathbf{x}) \ \delta_{\theta_0}(d \phi) \right ] \leq  \\
        & D \sqrt{k-1} \ \mathbb{E}_{\otimes^n \mathfrak{p}_0} \left [\int_{\Theta^2} \lVert \theta - \phi \rVert_{l^2} \ \psi_n(d \theta | \mathbf{X} = \mathbf{x}) \ \delta_{\theta_0}(d \phi) \right ]  \leq  \\
        & D \sqrt{k-1} \ \mathbb{E}_{\otimes^n \mathfrak{p}_0} \left [ \left ( \int_{\Theta^2} \lVert \theta - \phi \rVert_{l^2}^2 \ \psi_n(d \theta | \mathbf{X} = \mathbf{x}) \ \delta_{\theta_0}(d \phi) \right )^{1/2} \right ] =  \\
        & D\sqrt{k-1}\mathbb{E}_{\otimes^n \mathfrak{p}_0} \left[ \mathcal{W}_2^{\mathcal{P}(\Theta)}(\psi_n( \ \cdot \mid \mathbf{X}), \delta_{\theta_0})  \right] = O(n^{-1/2}),
    \end{aligned}
    \end{equation*}
    where the last equality is due to Proposition \ref{pcrPara}. Moreover, by Theorem \ref{fournier}, we have
    \begin{equation*}
    \mathbb{E}_{\psi_n \otimes \mathfrak{p}_0^n} \Bigl[
    \mathbb{E} \bigl[ \mathcal{W}_1^{\mathcal{P}(\mathbb{X})}(e_m^{(\mathbf{Z})}, \mathscr{T}(\theta)) 
    \mid \theta \bigr] \Bigr]
    =
    \begin{cases}
    O(m^{-1/2}) & \text{if } d < 2, \\[1mm]
    O(m^{-1/2} \log(1+m)) & \text{if } d = 2, \\[1mm]
    O(m^{-1/d}) & \text{if } d > 2.
    \end{cases}
    \end{equation*}

    since $Z_1,...,Z_m \mid \theta \iid \mathscr{T}(\theta)$ and a discrete measure with finite support has $q$-moments finite for all $q \geq 1$.

\subsection{Proof of Lemma \ref{concDP}}
We use a dyadic transport technique as in \cite{weed2017sharpasymptoticfinitesamplerates}, which allows to
compute the order of $\mathbb{E}_{\tilde p|X^{\left( n\right) }}\left( \mathcal{W}_1\left( P,H_{n}\right) \right) $, where $X^{(n)}=(X_1,\dots,X_n)$ and $P$ is a random measure distributed as $\tilde p |X^{(n)}$, which has distribution $\pi _{n}\left( \cdot |X^{\left( n\right) }\right) =\mathscr{D}%
\left( \theta +n,H_n\right) $.. 
Consider a dyadic partition $\left\{ 
\mathcal{Q}^{k}\right\} _{k=1}^{k^{\ast }}$ of $\mathbb{X}$ with parameter $%
\delta <1$: then Proposition 1 from \cite{weed2017sharpasymptoticfinitesamplerates} gives the inequality%
\begin{equation*}
\mathcal{W}_{1}\left( P,H_{n}\right) \leq \delta ^{k^{\ast }}+\sum_{k=1}^{k^{\ast
}}\delta ^{k-1}\sum_{Q_{i}^{k}\in \mathcal{Q}^{k}}\left\vert P\left(
Q_{i}^{k}\right) -H_{n}\left( Q_{i}^{k}\right) \right\vert ,
\end{equation*}

and notice we have $\sum_{Q_{i}^{k}\in \mathcal{Q}^{k}}\left\vert P\left(
Q_{i}^{k}\right) -H_{n}\left( Q_{i}^{k}\right) \right\vert =\left\vert
\left\vert \mathbf{P}^{k}-\mathbf{H}_{n}^{k}\right\vert \right\vert _{1}$, 
where 
$\mathbf{P}^{k}= \left(P\left( Q_{1}^{k}\right) ,...,P\left( Q_{m_{k}}^{k}\right) \right)
$, $\mathbf{H}_{n}^{k}=\left( H_{n}\left( Q_{1}^{k}\right) ,...,H_{n}\left(
Q_{m_{k}}^{k}\right) \right) $ and $m_{k}$ is the cardinality of the
partition $\mathcal{Q}^{k}$. Hence%
\begin{equation}\label{weedineq}
\mathcal{W}_{1}\left( P,H_{n}\right) \leq \delta ^{k^{\ast }}+\sum_{k=1}^{k^{\ast
}}\delta ^{k-1}\left\vert \left\vert \mathbf{P}^{k}-\mathbf{H}_{n}^{k}\right\vert \right\vert_{1} .
\end{equation}

$\mathbf{P}^{k}$ is known to be, conditionally on $X^{\left( n\right) }$, a random
vector with Dirichlet distribution, whose mean is the vector $H_{n}^{k}$.

If one considers a random vector $\mathbf{Y}\sim \text{Dirichlet}\left( \beta
_{1},...,\beta _{n}\right) $, defining $\mu _{i}=\frac{\beta _{i}}{\sum
\beta _{i}}=\frac{\beta _{i}}{\beta _{\cdot }}$ for $i=1,...n$, then 
\begin{equation*}
\begin{aligned}
\mathbb{E}\left( \left\vert \left\vert \mathbf{Y-\mu }\right\vert
\right\vert _{1}\right) 
&\leq \sqrt{n}\sqrt{\mathbb{E}\left( \left\vert
\left\vert \mathbf{Y-\mu }\right\vert \right\vert _{2}^{2}\right) }=\sqrt{n}\sqrt{\sum_{i=1}^{n}var\left( Y_{i}\right) }\\
&=\sqrt{n}\sqrt{%
\sum_{i=1}^{n}\frac{\mu _{i}\left( 1-\mu _{i}\right) }{\beta _{\cdot }+1}}=%
\sqrt{n}\sqrt{\frac{1}{\beta _{\cdot }+1}\left( 1-\sum \mu _{i}^{2}\right) }%
\leq \sqrt{n}\frac{1}{\sqrt{\beta _{\cdot }+1}}.
\end{aligned}
\end{equation*}

So in our case $\mathbb{E}%
_{\tilde p|X^{n}}\left( \left\vert \left\vert \mathbf{P}^{k}-\mathbf{H}_{n}^{k}\right\vert
\right\vert _{1}\right) \leq \sqrt{\frac{m_{k}}{\theta +n+1}}$, and then
taking the expectation of both sides in (\ref{weedineq}) yields%
\begin{equation*}
\mathbb{E}_{\tilde p|X^{n}}\left(\mathcal{W}_{1}\left( P,H_{n}\right) \right) \leq \delta
^{k^{\ast }}+\sum_{k=1}^{k^{\ast }}\delta ^{k-1}\mathbb{E}_{\tilde p|X^{n}}\left(
\left\vert \left\vert \mathbf{P}^{k}-\mathbf{H}_{n}^{k}\right\vert \right\vert _{1}\right)
\leq \delta ^{k^{\ast }}+\frac{1}{\sqrt{\theta +n+1}}\sum_{k=0}^{k^{\ast
}-1}\delta ^{k}\sqrt{m_{k}}.
\end{equation*}

Since we have that if $Q\in \mathcal{Q}^{k}$ then $\text{diam}\left( Q\right) \leq
\delta ^{k}$, it holds $m_{k}=O\left( \delta ^{-kd}\right) $. Then the previous inequality reads
\begin{equation*}
\mathbb{E}_{\tilde p|X^{n}}\left( \mathcal{W}_{1}\left( P,H_{n}\right) \right) \lesssim
\delta ^{k^{\ast }}+\frac{1}{\sqrt{\theta +n+1}}\sum_{k=0}^{k^{\ast
}-1}\delta ^{k}\delta ^{-\frac{kd}{2}}=\delta ^{k^{\ast }}+\frac{1}{\sqrt{%
\theta +n+1}}\sum_{k=0}^{k^{\ast }-1}\delta ^{k\left( 1-\frac{d}{2}\right) }.
\end{equation*}

If $d>2$, using $\sum_{k=0}^{k^{\ast }-1}\left( \delta ^{1-d/2}\right) ^{k}=%
\frac{\left( \delta ^{1-d/2}\right) ^{k^{\ast }}-1}{\delta ^{1-d/2}-1}$
yields the inequality
\begin{equation*}
\mathbb{E}_{\tilde p|X^{n}}\left( \mathcal{W}_{1}\left( P,H_{n}\right) \right) \lesssim
\delta ^{k^{\ast }}+\frac{1}{\sqrt{\theta +n+1}}\frac{\left( \delta
^{1-d/2}\right) ^{k^{\ast }}-1}{\delta ^{1-d/2}-1}. 
\end{equation*}

This holds for any integer $k^{\ast }$ and $\delta >0$, so we can optimize
the bound. If we take $\delta $ fixed and optimize in $t:=\delta ^{k^{\ast
}} $, then%
\begin{equation*}
\mathbb{E}_{\tilde p|X^{n}}\left( \mathcal{W}_{1}\left( P,H_{n}\right) \right) \lesssim t+%
\frac{1}{\sqrt{\theta +n}}t^{1-d/2} 
\end{equation*}

and the optimization yields $t^{\frac{d}{2}}=C^{\prime }\frac{1}{\sqrt{%
\theta +n}}$, i.e. $t=C^{\prime }\left( \theta +n\right) ^{-\frac{1}{d}}$.
This entails that $\mathbb{E}_{\tilde p|X^{n}}\left( \mathcal{W}_{1}\left( P,H_{n}\right)
\right) =O\left( n^{-\frac{1}{d}}\right) $.

If $d=2$,%
\begin{equation*}
\mathbb{E}_{\tilde p|X^{n}}\left( \mathcal{W}_{1}\left( P,H_{n}\right) \right) \lesssim
\delta ^{k^{\ast }}+\frac{k^{\ast }}{\sqrt{\theta +n+1}}. 
\end{equation*}

Optimization in $k^{\ast }$ gives $\delta ^{k^{\ast }}\log \delta +\frac{1}{%
\sqrt{\theta +n+1}}=0$, i.e. $\delta ^{k^{\ast }}=-\frac{1}{\log \delta }%
\frac{1}{\sqrt{\theta +n+1}}$, so that $k^{\ast }=O\left( \log n\right) $
and $\mathbb{E}_{\tilde p|X^{n}}\left( \mathcal{W}_{1}\left( P,H_{n}\right) \right) =O\left( 
\frac{\log n}{\sqrt{n}}\right) $.

If $d<2$, 
\begin{equation*}
\mathbb{E}_{\tilde p|X^{n}}\left( \mathcal{W}_{1}\left( P,H_{n}\right) \right) \lesssim
\delta ^{k^{\ast }}+\frac{1}{\sqrt{\theta +n+1}}\frac{\left( \delta
^{1-d/2}\right) ^{k^{\ast }}-1}{\delta ^{1-d/2}-1}\leq \delta ^{k^{\ast }}+c\frac{1}{\sqrt{\theta +n+1}},
\end{equation*}

so the optimal order is obtained for $\delta ^{k^{\ast }}=O\left( \frac{1}{%
\sqrt{n}}\right) $.

In all three optimization procedures, ensuring that $k^{\ast }$ is integer
by taking the ceiling or floor function does not affect the order of
convergence.

\subsection{Proof of Theorem \ref{consistencydir}}
If $f:\mathbb{X}^{m}\rightarrow \lbrack 0,+\infty )$ is measurable, then, by definition of $\nu _{n,m}$,
\begin{equation}\label{firsteq}
\begin{aligned}
\mathbb{E}_{\nu _{n,m}}\left( f\left( \mathbf{Z}\right) \right) &=\int_{%
\mathbb{X}^{m}\times \mathbb{X}^{n}}f\left( z_{1},...,z_{m}\right)
Q_{n}\left( dz_{1},...,dz_{m}|\mathbf{X=x}\right) \mathfrak{p}_{0}^{n}\left(
d\mathbf{x}\right) \\
&=\int_{\mathbb{X}^{m}\times \mathbb{X}^{n}}f\left( z_{1},...,z_{m}\right)
\left( \int_{\mathcal{P}\left( \mathbb{X}\right) }\prod_{j=1}^{m}\tilde{p}%
\left( dz_{j}\right) \pi _{n}\left( d\tilde{p}|\mathbf{x}\right) \right) 
\mathfrak{p}_{0}^{n}\left( d\mathbf{x}\right) \\
&=\int_{\mathcal{P}\left( \mathbb{X}\right) \times \mathbb{X}^{n}}\left( \int_{%
\mathbb{X}^{m}}f\left( z_{1},...,z_{m}\right) \prod_{j=1}^{m}\tilde{p}\left(
dz_{j}\right) \right) \pi _{n}\left( d\tilde{p}|\mathbf{x}\right) \mathfrak{p%
}_{0}^{n}\left( d\mathbf{x}\right) \\
&=\int_{\mathcal{P}\left( \mathbb{X}\right) \times \mathbb{X}^{n}}\mathbb{E}%
\left( f\left( \mathbf{Z}\right) |\tilde{p}\right) \pi _{n}\left( d\tilde{p}|%
\mathbf{x}\right) \mathfrak{p}_{0}^{n}\left( d\mathbf{x}\right),
\end{aligned}
\end{equation}
where the third equality follows from Fubini-Tonelli's theorem. Now, applying the equality (\ref{firsteq}) to  $%
f\left( \mathbf{Z}\right) =\mathcal{W}_{1}\left( e_{m}^{\left( \mathbf{Z}\right) },\mathfrak{p}_0\right)$ and using the
triangle inequality on $\tilde{p}$, we get
\begin{equation}\label{firstrieq}
\begin{aligned}
\mathbb{E}_{\nu _{n,m}}\left( \mathcal{W}_{1}
\left( e_{m}^{\left( \mathbf{Z}\right) },\mathfrak{p}_{0}\right)\right) 
\leq\;&\underbrace{\int_{\mathcal{P}\left( \mathbb{X}\right) \times \mathbb{X}^{n}}
\mathbb{E}\left( \mathcal{W}_{1}
\left( e_{m}^{\left( \mathbf{Z}\right) },\tilde{p}\right) \mid \tilde{p}\right)
\pi _{n}\left( d\tilde{p}\mid\mathbf{x}\right)\mathfrak{p}_{0}^{n}\left( d\mathbf{x}\right)}_{A} \\
&\quad+\underbrace{\int_{\mathcal{P}\left( \mathbb{X}\right) \times \mathbb{X}^{n}}
\mathbb{E}\left( \mathcal{W}_{1}
\left( \tilde{p},\mathfrak{p}_{0}\right) \mid \tilde{p}\right)
\pi _{n}\left( d\tilde{p}\mid\mathbf{x}\right) \mathfrak{p}_{0}^{n}\left( d\mathbf{x}\right)}_{B}.
\end{aligned}
\end{equation}

We bound separately $A$ and $B$.
\par\bigskip
$A$, which we define as $\varepsilon _{m,1}\left( \mathbb{X}%
\right) $, is the rate of convergence for the mean of the 1-Wasserstein
distance between $e_{m}^{\left( \mathbf{Z}\right) }$ and a random
probability measure with law $\pi _{n}$, posterior of a sample from the
Dirichlet prior (up to the integral with respect to $\mathfrak{p}%
_{0}^{n}\left( d\mathbf{x}\right) $, which does not change the rate). If we show that a random probability measure distributed as the posterior of a Bayesian model with Dirichlet process prior has $q$-moments finite for some integer $q > 1$, at least $\pi_n$-almost surely, applying Theorem \ref{fournier} under the assumption $\mathbb{X}\subseteq \mathbb{R}^d$ we get the second part of the thesis. 
By Proposition \ref{pitman posterior}, $\tilde{p} \sim \pi_n( \ \cdot \mid X_1,\dots,X_n)$ implies
\begin{equation*}
    \tilde{p} \overset{d}{=} \sum_{i=1}^j W_i \delta_{X_i^*} + R_j F_j,
\end{equation*}
where $(W_1,...,W_j,R_j) \sim \text{Dirichlet}(n_1,...,n_j, \theta)$, independently of the random distribution $F_j$, which has $\mathscr{D}(\theta,H)$ distribution. Moreover, by stick-breaking construction of the DP (definition of Pitman--Yor process with $\sigma = 0$), we have
\begin{equation*}
    F_j = \sum_{i=1}^\infty Y_i \delta_{\Phi_i},
\end{equation*}
where $\Phi_i \iid H$, and 
\begin{equation*}
    Y_1 = V_1, \quad Y_l = V_l \prod_{j = 1}^{l-1} (1 - V_j), \quad l \geq 2,
\end{equation*}
with $\{V_i\}_{i \geq 1}$ sequence of independent random variables taking values in $[0,1]$ and distributed as a $\text{Beta}(1,\theta)$. The following holds:
\begin{equation*}
    \int_{\mathbb{X}} |x|^q \ \tilde{p}(d x) =  \sum_{i=1}^j W_i |X_i^*|^q + R_j \sum_{i = 1}^\infty Y_i |\Phi_i|^q \leq 2 R^q, \qquad \pi_n\text{-a.s.}
\end{equation*}
where $R > 0 $ is the radius of the ball that contains the space $\mathbb{X}$, which exists because $\mathbb{X}$ is totally bounded. Since we have finite moments for any $q > 1$, we can neglect the terms involving $q$ in the bounds of $\varepsilon_{m,1}(\mathbb{X})$. \par\bigskip
Let us turn to $B$. Since $\mathbb{E}\left( \mathcal{W}_{1}\left( \tilde{p},\mathfrak{p}_{0}\right) |\tilde{p}%
\right) =\mathcal{W}_{1}\left( \tilde{%
p},\mathfrak{p}_{0}\right) $ by conditional determinism, we have 
\begin{align*}
\int_{\mathcal{P}\left( \mathbb{X}\right) \times \mathbb{X}^{n}}\mathbb{E}%
\left( \mathcal{W}_{1}\left( \tilde{p},\mathfrak{p}_{0}\right) |\tilde{p}\right) \pi _{n}\left( d\tilde{p}|\mathbf{%
x}\right) \mathfrak{p}_{0}^{n}\left( d\mathbf{x}\right) &=\int_{\mathcal{P}%
\left( \mathbb{X}\right) \times \mathbb{X}^{n}}\mathcal{W}_{1}\left( \tilde{p},\mathfrak{p}_{0}\right) \pi
_{n}\left( d\tilde{p}|\mathbf{x}\right) \mathfrak{p}_{0}^{n}\left( d\mathbf{x}\right)\\
&=\mathbb{E}_{\mathfrak{p}_{0}^{n}}\left( \mathbb{E}_{\tilde p|X^{\left( n\right) }}\left( \mathcal{W}_{1}\left(P,\mathfrak p_0 \right) \right) \right),
\end{align*}

where $P$ distributed as $\tilde p |X^{(n)} \sim \pi _{n}\left( \cdot |X^{\left( n\right) }\right) =\mathscr{D}%
\left( \theta +n,H_n\right) $. We have
\begin{equation}\label{secondtrieq}
\begin{aligned}
\mathbb{E}_{\mathfrak{p}%
_{0}^{n}}\left( \mathbb{E}_{\tilde p|X^{\left( n\right) }}\left( \mathcal{W}_{1}\left(
P,\mathfrak{p}_{0}\right) \right) \right) \leq \mathbb{E}_{\mathfrak{p}_{0}^{n}}\left( 
\mathcal{W}_{1}\left( \mathfrak{p}_{0},H_{n}\right) \right) +\mathbb{E}_{\mathfrak{p}_{0}^{n}}\left( \mathbb{E}_{\tilde p|X^{\left( n\right) }}\left( \mathcal{W}_{1}\left(P,H_{n}\right) \right) \right).
\end{aligned}
\end{equation}

Concerning the first term on the right hand side of (\ref{secondtrieq}), convexity of $\mathcal{W}%
_{1}$ yields 
\begin{equation*}
\mathcal{W}_{1}^{\ }\left( H_{n},\mathfrak{p}_{0}\right) =\mathcal{W}_{1}^{\
}\left( \frac{\theta }{\theta +n}H+\frac{n}{\theta +n}e_{n}^{\left( \mathbf{X%
}\right) },\mathfrak{p}_{0}\right) \leq \frac{\theta }{\theta +n}\mathcal{W}%
_{1}^{\ }\left( H,\mathfrak{p}_{0}\right) +\frac{n}{\theta +n}\mathcal{W}
_{1}^{\ }\left( e_{n}^{\left( \mathbf{X}\right) },\mathfrak{p}_{0}\right). 
\end{equation*}

$\mathcal{W}_{1}^{\ }\left( H,\mathfrak{p}_{0}\right) $ is a constant
independent on $n$, and $\frac{\theta }{\theta +n}=O\left( \frac{1}{n}%
\right) $. Moreover, Theorem \ref{fournier} states that
\begin{equation*}
\mathbb{E}_{\mathfrak{p}_0^n} \Bigl( \mathcal{W}_1 \bigl(
e_n^{(\mathbf{X})}, \mathfrak{p}_0 \bigr) \Bigr)
=
\begin{cases}
O(n^{-1/2}) & \text{if } d<2,\\
O(n^{-1/2} \log(1+n)) & \text{if } d=2,\\
O(n^{-1/d}) & \text{if } d>2.
\end{cases}
\end{equation*}

It follows that
\begin{equation*}
\mathbb{E}_{\mathfrak{p}_0^n} \Bigl( \mathcal{W}_1 \bigl(
\mathfrak{p}_0, H_n \bigr) \Bigr)
=
\begin{cases}
O(n^{-1/2}) & \text{if } d<2,\\
O(n^{-1/2} \log(1+n)) & \text{if } d=2,\\
O(n^{-1/d}) & \text{if } d>2
\end{cases}
\end{equation*}

A rate for the second term on the right hand side of (\ref{secondtrieq}) is given by Lemma \ref{concDP}. Summing the two rates we obtain 
\begin{align*}
\mathbb{E}_{\nu_{n,m}}\left[ \mathcal{W}_1
\left( e_m^{(\mathbf{Z})}, \mathfrak{p}_0 \right) \right]
&= \varepsilon_{m,1}(\mathbb{X})
+ 
\begin{cases}
O(n^{-1/2}) & \text{if } d<2,\\
O(n^{-1/2} \log(1+n)) & \text{if } d=2,\\
O(n^{-1/d}) & \text{if } d>2,
\end{cases}
+ 
\begin{cases}
O(n^{-1/2}) & \text{if } d<2,\\
O\left(\frac{\log n}{\sqrt{n}}\right) & \text{if } d=2,\\
O(n^{-1/d}) & \text{if } d>2.
\end{cases} \\[1mm]
&=
\begin{cases}
O(n^{-1/2}) & \text{if } d<2,\\
O\left(\frac{\log n}{\sqrt{n}}\right) & \text{if } d=2,\\
O(n^{-1/d}) & \text{if } d>2.
\end{cases}
\end{align*}
This finishes the proof.

Notice that we have been computing a rate for a $1$-Wasserstein posterior contraction rate for the Dirichlet prior. Indeed, it holds 
\[
\mathbb{E}_{\mathfrak{p}_{0}^{n}}\left( \mathbb{E}_{\tilde p|X^{\left( n\right) }}\left( \mathcal{W}_{1}\left(P,\mathfrak p_0 \right) \right) \right) = \mathbb{E}_{\mathfrak{p}_{0}^{n}}\left( \mathcal{W}_{1}^{\mathcal{P}(\mathcal P \left( 
\mathbb{X})\right) }\left( \pi_{n}\left( \cdot |X^{\left( n\right) }\right)
,\delta _{\mathfrak{p}_{0}}\right) \right)
\]
because
\begin{equation*}
\begin{aligned}
\mathcal{W}_{1}^{\mathcal{P}(\mathcal P\left( \mathbb{X})\right) }\left( \pi _{n}\left(
\cdot |X^{\left( n\right) }\right) ,\delta _{\mathfrak{p}_{0}}\right)
&=\inf_{\eta \in \Pi\left( \pi _{n}\left( \cdot |X^{\left( n\right)
}\right) ,\delta _{\mathfrak{p}_{0}}\right) }\int_{\mathcal{P}\left( \mathbb{%
X}\right) \times \mathcal{P}\left( \mathbb{X}\right) }\mathcal{W}_{1}\left( p_{1},p_{2}\right) \eta \left(
dp_{1}dp_{2}\right) \\
&=\int_{\mathcal{P}\left( \mathbb{X}\right) }\mathcal{W}%
_{1}\left( p_{1},\mathfrak{p}%
_{0}\right) \pi _{n}\left( dp_{1}|X^{\left( n\right) }\right) ,
\end{aligned}
\end{equation*}
since the only coupling with a delta measure is the independent coupling.
Therefore 
\begin{equation*}
\mathbb{E}_{\mathfrak{p}_{0}^{n}}\left( \mathcal{W}_{1}^{\mathcal{P}\left( 
\mathcal P(\mathbb{X})\right) }\left( \pi _{n}\left( \cdot |X^{\left( n\right) }\right)
,\delta _{\mathfrak{p}_{0}}\right) \right) =\int_{\mathbb{X}^{n}}\left(
\int_{\mathcal{P}\left( \ \mathbb{X}\right) \ }\mathcal{W}_{1}\left( p_{1},\mathfrak{p}_{0}\right) \pi
_{n}\left( dp_{1}|\mathbf{x}\right) \right) \mathfrak{p}_{0}^{n}\left( d%
\mathbf{x}\right). 
\end{equation*}

\section{Details on numerical illustrations}\label{details_num}

We provide additional details on the experiments in Section \ref{downstream}, Section \ref{conv}, and Section \ref{cal_data}.

\subsection{Details on downstream Bayesian inference}

The downstream-inference experiment has two components. The first is a timing sanity check for the exact mechanism-aware posterior associated with the perturbed-histogram synthetic-sample release. In this exact formulation, the MCMC state contains the full latent confidential sample \(X_{1:n}\), together with the noisy histogram variables used internally by the perturbed histogram mechanism. At each saved iteration, posterior inference for \(\tilde p\) is performed from the closed-form conditional posterior
\[
\tilde p \mid X_1,\ldots,X_n \sim \mathscr D\left(\theta+n,\frac{\theta H+\sum_{i=1}^{n}\delta_{X_i}}{\theta+n}\right).
\]
For \(n=10^4\), \(m=10^2\), and \(25\) histogram bins, the reference implementation takes about six minutes to obtain \(1000\) posterior samples. This timing is reported only to document the computational scale of the exact posterior. We do not use this sampler in repeated Monte Carlo experiments or in the real-data analysis.

The second component is the repeated fallback experiment reported in Section \ref{downstream}. The confidential data are generated from a distribution supported on a regular grid in \([0,1]\). The grid probabilities are obtained by evaluating a two-component beta mixture on the grid and normalizing the resulting weights. This produces a non-uniform discrete distribution with asymmetric shape and non-negligible tail mass. We generate \(n=2000\) confidential observations and compare three downstream workflows over \(20\) independent Monte Carlo replicates.

For the Dirichlet-process posterior-predictive mechanism, we release \(m\) synthetic observations and compute posterior summaries from the ordinary Dirichlet-process posterior given the released sample. For the direct perturbed histogram, the observed object is the noisy histogram output, and the mechanism-aware posterior is explored by MCMC over the latent confidential bin counts. For the direct smoothed histogram, the output is treated as a computational surrogate. As discussed in Section \ref{downstream}, this direct smoothed output is not differentially private, except in the degenerate case of complete smoothing, and is therefore included only as a non-private diagnostic.

Table \ref{tab:downstream_summary} summarizes the downstream-inference results. We report bias, root mean square error, average credible-interval length, and empirical coverage for the posterior mean, quartiles, and right-tail probability \(\tilde p((0.75,1])\). We also report wall-clock time and effective sample size per second when MCMC is used. The main qualitative conclusion is that the Dirichlet-process posterior-predictive mechanism has statistical accuracy comparable to the histogram-based workflows in this simplified setting, while avoiding mechanism-aware latent-data or latent-count MCMC.

\begin{table}[htbp]
\centering
\tiny

\begin{tabular}{llrrrrrr}
\toprule
\textbf{Method} & \textbf{Functional} &
\textbf{Bias} & \textbf{RMSE} &
\textbf{Coverage} & \textbf{CI length} &
\textbf{\makecell[b]{Average elapsed \\ seconds}} &
\textbf{ESS/s} \\
\midrule

\multirow{5}{*}{\makecell[l]{Dirichlet-process \\ posterior-predictive \\ mechanism}} & mean      &  0.0018 & 0.0125 & 1.00 & 0.0525 & 0.021 & 38238.1 \\
 & median    &  0.0017 & 0.0156 & 1.00 & 0.0665 & 0.021 & 37809.5 \\
 & Q1 (25\%)       &  0.0008 & 0.0110 & 0.85 & 0.0389 & 0.021 & 37877.7 \\
 & Q3 (75\%)       & -0.0058 & 0.0285 & 1.00 & 0.1382 & 0.021 & 37685.2 \\
 & $\tilde p((0.75,1])$ &  0.0004 & 0.0162 & 1.00 & 0.0709 & 0.021 & 38237.2 \\
\midrule

\multirow{5}{*}{\makecell[l]{Direct perturbed \\ histogram \\ mechanism}} & mean
& -0.0010 & 0.0055 & 1.00 & 0.0264 & 1.271 & 607.4 \\
 & median
&  0.0031 & 0.0105 & 1.00 & 0.0683 & 1.271 & 602.7 \\
 & Q1 (25\%)
& -0.0097 & 0.0109 & 1.00 & 0.0613 & 1.271 & 591.0 \\
 & Q3 (75\%)
& -0.0042 & 0.0135 & 1.00 & 0.0832 & 1.271 & 598.8 \\
 & $\tilde p((0.75,1])$
& -0.0024 & 0.0069 & 1.00 & 0.0660 & 1.271 & 571.7 \\
\midrule

\multirow{5}{*}{\makecell[l]{Direct smoothed \\ histogram \\ mechanism}} & mean
& -0.0009 & 0.0054 & 1.00 & 0.0254 & 0.059 & 13167.7 \\
 & median
&  0.0035 & 0.0104 & 1.00 & 0.0678 & 0.059 & 13102.8 \\
 & Q1 (25\%)
& -0.0099 & 0.0110 & 1.00 & 0.0602 & 0.059 & 13163.9 \\
 & Q3 (75\%)
& -0.0040 & 0.0136 & 1.00 & 0.0810 & 0.059 & 13131.0 \\
 & $\tilde p((0.75,1])$
& -0.0021 & 0.0071 & 1.00 & 0.0663 & 0.059 & 12914.6 \\

\bottomrule
\end{tabular}

\caption{Summary of downstream inference performance}\label{tab:downstream_summary}
\end{table}

\subsection{Details on the convergence simulation}

In Section \ref{conv}, the simulated confidential data are generated as follows. Let \(G\sim\mathrm{Geometric}(0.05)\), with support \(\{1,2,\ldots\}\), and define \(X=T(G)\), where
\[
T(k)=\frac{2(k-2^{r(k)-1})+1}{2^{r(k)}},\qquad r(k)=\lfloor \log_2 k \rfloor+1.
\]
This map sends the positive integers to dyadic midpoints in \([0,1]\) by increasing dyadic resolution. In particular, \(1\mapsto 1/2\), \(2\mapsto 1/4\), \(3\mapsto 3/4\), \(4\mapsto 1/8\), and so on. The construction yields a discrete distribution on \([0,1]\) with repeated values and a long tail of increasingly fine dyadic categories.

Figure \ref{w1_fig} is obtained from a simulated population of size \(N=1000000\). For each \(n\in\{10000,50000,75000,100000,200000,250000,500000,1000000\}\), we draw confidential subsamples of size \(n\), generate synthetic data using the Dirichlet-process posterior-predictive mechanism, and compute \(\mathcal W_1\left(e_m^{(\mathbf Z)},e_n^{(\mathbf X)}\right)\). The reported values are averages over \(100\) independent runs, with \(99\%\) confidence bands.

\subsection{Details on the California census dataset}

In Section \ref{cal_data}, we use the PINCP variable from the \(2024\) \(5\)-Year Public Use Microdata Sample from the American Community Survey. Observations with missing or non-positive income are removed. The remaining values are transformed by taking logarithms and then applying min-max scaling, so that the final data lie in \([0,1]\). The cleaned and transformed dataset contains \(11918162\) observations, \(32393\) distinct recorded values, and \(7689\) singletons.

The Dirichlet-process posterior-predictive mechanism uses \(H=\operatorname{Unif}[0,1]\), \(\theta=1\), \(\varepsilon=2\), and \(\delta=10^{-5}\). The released sample size is chosen as the largest integer satisfying the privacy bound in Theorem \ref{cor_dp}, which gives \(m=119\). The histogram mechanisms use \(300\) bins and the same released sample size. The perturbed histogram uses Gaussian noise calibrated to the \((\varepsilon,\delta)\)-differential privacy target. The smoothed histogram is composed with an \((\varepsilon,\delta)\)-differentially private Gaussian perturbation of the histogram counts. The Wasserstein boxplots in Figure \ref{fig:wass_boxplots} are based on \(100\) independent runs of each synthetic-data mechanism.

\end{document}